\documentclass[11pt]{article}
\usepackage{}
\usepackage{enumerate}
\usepackage{mathbbold}
\usepackage{bbm}
\usepackage{amsfonts}
\usepackage[T1]{fontenc}
\usepackage{latexsym,amssymb,amsmath,amsfonts,amsthm}
\usepackage{graphicx, color}
\usepackage{subfigure}
\usepackage{epsf}
\usepackage{epstopdf}
\usepackage{amssymb}
\usepackage{mathrsfs}

\renewcommand{\theequation}{\arabic{section}.\arabic{equation}}

\newcommand{\R}{{\mathbb R}}
\newcommand{\Z}{{\mathbb Z}}

\newcommand{\C}{{\mathbb C}}
\newcommand{\Sp}{{\mathbb S}}
\newcommand{\ds}{\displaystyle}

\newcommand{\be}{\begin{eqnarray}}
\newcommand{\ben}{\begin{eqnarray*}}
\newcommand{\en}{\end{eqnarray}}
\newcommand{\enn}{\end{eqnarray*}}
\newcommand{\ba}{\backslash}
\newcommand{\pa}{\partial}

\newcommand{\ov}{\overline}

\newcommand{\ddiv}{{\rm div\,}}

\newcommand{\g}{\gamma}
\newcommand{\G}{\Gamma}

\newcommand{\om}{\omega}

\newcommand{\la}{\lambda}

\newcommand{\wi}{\widetilde}

\newcommand{\hth}{\theta}
\newcommand{\hx}{\hat{x}}
\newtheorem{theorem}{Theorem}[section]

\newtheorem{proposition}[theorem]{Proposition}

\begin{document}
\title{\bf Inverse acoustic scattering problems with multi-frequency sparse backscattering far field data}
\author{Xia Ji\thanks{LSEC, NCMIS and Academy of Mathematics and Systems Science, Chinese Academy of Sciences,
Beijing 100190, China. Email: jixia@lsec.cc.ac.cn (XJ)},
\and
Xiaodong Liu\thanks{NCMIS and Academy of Mathematics and Systems Science,
Chinese Academy of Sciences, Beijing 100190, China. Email: xdliu@amt.ac.cn (XL)}}
\date{}
\maketitle

\begin{abstract}
The inverse acoustic scattering problems using multi-frequency backscattering far field patterns at isolated directions are studied.
The underlying object could be point like scatterers, small scatterers, extended inhomogeneities and obstacles.
A fast and robust direct sampling method is proposed for location and shape reconstruction of the underlying objects.
To our best knowledge, this is the first numerical method for inverse obstacle/medium scattering problems with the multi-frequency sparse backscattering data.
Some approximate uniqueness results have also been showed based on the weak scattering approximation and the Kirchhoff approximation. The theoretical basis of the direct sampling method is initially established based on these two approximations.
Numerical examples in two dimensions are presented to validity the effectiveness and robustness of the proposed direct sampling method.
In particular, the numerical reconstructions indicate that the direct sampling method also works very well beyond these two approximations. 

\vspace{.2in}

\vspace{.2in} {\bf AMS subject classifications:}
35P25, 45Q05, 78A46, 74B05

\end{abstract}

\section{Introduction}

The inverse scattering theory has been a fast-developing area for the past forty years.
Applications of inverse scattering problems occur in many areas such as radar, nondestructive testing,
medical imaging, geophysical prospection and remote sensing. We refer to the standard monograph \cite{CK} for a research statement on the significant progress
both in the mathematical theories and the numerical approaches.

Majority of studies focuses on inverse time harmonic wave scattering problems at a fixed frequency.
Uniqueness for the inverse scattering problems can be established if the measurements are taken for all observation directions and all
incident directions.
Besides the traditional iterative methods, many non-iterative methods have also been proposed for shape reconstructions, see e.g.,
the well developed linear sampling method \cite{ColtonKirsch}, the factorization method \cite{Kirsch98,KirschGrinberg},
and various types of recently proposed direct sampling methods \cite{CCHuang,ItoJinZou,LiLiuZou,LiZou,LiuIP17, Potthast2010}.
In particular, the direct sampling methods inherit many advantages of the classical linear sampling method and factorization method, e.g., they are independent of any a priori information on the geometry and physical properties of the unknown objects.
The main feature of these direct sampling methods is that only the inner product of the measurements with some suitably chosen functions is involved in
the computation of the indicator. Thus these direct sampling methods are robust to noises and computationally faster than the classical sampling methods.
We refer to \cite{CCHuang,L4,LiLiuZou,LiZou,LiuIP17} for recent developments in this direction. In particular, their connection with the classical sampling methods has been
pointed out in \cite{LiuIP17}. However, up to now, the mathematical basis for the direct sampling methods is far less developed than the classical sampling methods.
The basic theoretical basis is the well known Funk-Hecke formula or the Helmholtz-Kirchhoff identity.
Thus, the same as the classical sampling methods, full-aperture data are needed for computation.
However, it is difficult to conduct an experiment to take measurements in all observation directions around
an unknown scatterer. From the practical point of view, we have only limited aperture data \cite{LiuSun}.
In limited aperture problem, the backscattering scenario is of particular interest, where one receiver and one
transmitter with a fixed location is used to collect data. The inverse backscattering problem has attracted the
attention of many researchers \cite{Bojarski,Devaney,EskinRalston2,EskinRalston3,HaddarKusiakSylvester,KR99,Langenberg,LiLiu,StefanovUhlmann}.

To make the inverse backscattering problem solvable, measurements should be taken with multiple frequencies.
Actually, in the last two decades, different multi-frequency methods have been proposed for inverse scattering problem.
These methods can be classified into two categories: iterative methods and direct methods.
The recursive linearization method (RLM) is an iterative method, which proceed via a continuation procedure with respect to frequency from low to high
\cite{BaoHouLi,BaoLiLinTriki,BaoLinTriki2, Chen, SiniThanh}.
An important feature of the RLM is that a fine reconstruction can be obtained without the need of a good initial guess.
A survey on the state of the art of the RLM can be found in \cite{BaoLiLinTriki}.
There are also many direct multi-frequency methods without using direct solvers, see e.g., the MUSIC (MUltiple-SIgnal-Classification) algorithm \cite{GriesmaierSchmiedecke} for locating small inhomogeneities, the Fourier method \cite{ZhangGuo},
the multi-frequency factorization method \cite{GriesmaierSchmiedecke-source}, and the direct sampling method \cite{AlaHuLiuSun} for source reconstructions,
the multi-frequency linear sampling method \cite{GuzinaCakoniBellis}, the eigenvalue method \cite{Sun2012IP} and an eigenfunction based scheme \cite{Liu2Wang2}
for obstacle reconstructions.

This paper is dedicated to a direct sampling method for inverse acoustic problem with multi-frequency sparse backscattering far field patterns.
The frequencies are located in some bounded band and the observation directions are sparse.
Of particular interest is what kind of information of the underlying scatterer can be recovered from the measurement at one or two observation directions.
To our best knowledge, this is the first numerical method for target reconstructions with multi-frequency sparse backscattering data.
Our method is motivated by a recent work \cite{AlaHuLiuSun},
where a direct sampling
method for source support reconstruction is proposed. The key observation for the source problem is that the far field measurement is just the Fourier
transform of the unknown source term.
We also refer to \cite{JL-elastic,JL-electromagnetic,JiLiuZhang-source} for the extensions to the inverse elastic and electromagnetic source problems
with multi-frequency sparse data.
Difficulties arise for the inverse obstacle/medium problems due to the nonlinearity between the measurement and the unknown scatterer.
To overcome the nonlinearity, we consider two linearized approaches: the weak scattering approximation and the Kirchhoff approximation.
The weak scattering approximation includes the point like scatterers, the small scatterers and the low contrast inhomogeneities.
The Kirchhoff approximation is applied for the extended obstacles. What we want to emphasize is that these two approximations are needed only for theoretical analysis. The numerical simulations show that our direct sampling method also works very well beyond these two approximations.

The remaining part of the work is organized as follows.
In the next section, we introduce three basic types of scattering objects: point like scatterers, penetrable inhomogeneities and extended obstacles.
We also introduce the corresponding inverse problem with three kinds of multi-frequency sparse backscattering far field data sets.
We then proceed in the Section 3 for some approximation uniqueness results based on the weak scattering approximation and the Kirchhoff approximation.
Section 4 is devoted to a direct sampling method with multi-frequency sparse backscattering data. The theory basis is also established.
The direct sampling method is then verified in Section \ref{NumExamples} by extensive examples in two dimensions.

\section{Inverse acoustic scattering with sparse backscattering data}\label{ExtendedObjects}

We begin with the formulations of the acoustic scattering problem. Let $k=\om/c>0$ be the wave number of
a time harmonic wave, where $\om>0$ and $c>0$ denote the frequency and sound
speed, respectively. In the whole paper, we consider multiple frequencies in a bounded band, i.e.,
\ben
k\in (k_{-}, k^{+}),
\enn
with two positive wave numbers $k_{-}$ and $k^{+}$.
Furthermore, let the incident field $u^{in}$ be a plane wave of the form
\be\label{incidenwave}
u^{in}(x)\ =\ u^{in}(x,\hth,k) = e^{ikx\cdot \theta},\quad x\in\R^n,
\en
where $\theta\in\Sp^{n-1}$ denotes the direction of the incident wave and $\Sp^{n-1}:=\{x\in\R^n:|x|=1\}$ is
the unit sphere in $\R^n$.

The three basic scattering objects are point like scatterers, penetrable medium and extended obstacle.

{\bf Point like scatterers.} The first case of our interest is the scattering by $M$ point like scatterers located at $z_1, z_2,\cdots, z_M\in\R^{n}$ in the homogeneous space $\R^{n},n=2,3$.
Recall the fundamental solution $\Phi(x,y), x,y\in \R^n, x\neq y,$ of the Helmholtz equation, which is given by
\be\label{Phi}
\Phi_k(x,y):=\left\{
         \begin{array}{ll}
         \ds\frac{ik}{4\pi}h^{(1)}_0(k|x-y|)=\frac{e^{ik|x-y|}}{4\pi|x-y|}, & n=3, \\
         \ds\frac{i}{4}H^{(1)}_0(k|x-y|), & n=2.
         \end{array}
         \right.
\en
Here, $h^{(1)}_0$ and $H^{(1)}_0$ are, respectively, spherical Hankel function and Hankel function of
the first kind and order zero.
By neglecting all the multiple scattering between the scatterers, the scattered field $u^{s}$ is approximately given by \cite{Foldy}
\be\label{uszm}
u^{s}(x,\hth,k)\approx\sum_{m=1}^{M}\tau_m u^{i}(z_m,\hth,k)\Phi_k(x,z_m).
\en
Here, $\tau_m\in\C$ is the scattering strength of the $m$-th target, $m=1,2,\cdots, M.$

{\bf Penetrable medium.}
Let $D\subset\R^n (n=2,\, 3)$ be an open and bounded domain with
Lipschitz boundary $\pa D$ such that the exterior $\R^n\ba\ov{D}$ is connected.
The simplest scattering of plane waves by penetrable medium is to find the total field $u=u^{in}+u^s$ such that
\be
\label{HemEqumedium}\Delta u + k^2 q u = 0\quad \mbox{in }\R^n,\\
\label{Srcmedium}\lim_{r:=|x|\rightarrow\infty}r^{\frac{n-1}{2}}\left(\frac{\pa u^{s}}{\pa r}-iku^{s}\right) =\,0,
\en
where $q=c_0^2/c^2$ is the refractive index given by the ratio of the square of the sound speeds, $c=c_0$ in the homogeneous background medium and $c=c(x)$
in the inhomogeneous medium. It is assumed that $q-1$ has compact support, which is denoted by $\ov{D}$.

{\bf Extended obstacle.}
The scattering of plane waves by impenetrable obstacle $D$ is to find the total field $u=u^{in}+u^s$ such that
\be
\label{HemEquobstacle}\Delta u + k^2 u = 0\quad \mbox{in }\R^n\ba\ov{D},\\
\label{Bc}\mathcal{B}(u) = 0\quad\mbox{on }\pa D,\\
\label{Srcobstacle}\lim_{r:=|x|\rightarrow\infty}r^{\frac{n-1}{2}}\left(\frac{\pa u^{s}}{\pa r}-iku^{s}\right) =\,0,
\en
where $\mathcal{B}$ denotes one of the following three boundary conditions
\ben
(1)\,\mathcal{B}(u):=u\quad\mbox{on}\, \pa D;\qquad
(2)\,\mathcal{B}(u):=\frac{\pa u}{\pa\nu}\quad\mbox{on}\ \pa D;\qquad
(3)\,\mathcal{B}(u):=\frac{\pa u}{\pa\nu}+i\la u\quad\mbox{on}\ \pa D
\enn
corresponding to the case when the scatterer $D$ is sound-soft, sound-hard and impedance type, respectively.
Here, $\nu$ is the unit outward normal to $\pa D$ and $\la$ is a positive constant.

The well-posedness of the direct scattering problems \eqref{HemEqumedium}--\eqref{Srcmedium} and \eqref{HemEquobstacle}--\eqref{Srcobstacle}
have been established and can be found in \cite{CK,Mclean}.
Every radiating solution of the Helmholtz equation has the following asymptotic
behavior at infinity \cite{KirschGrinberg, LiuIP17}
\be\label{0asyrep}
u^s(x,\theta,k)
=\frac{e^{i\frac{\pi}{4}}}{\sqrt{8k\pi}}\left(e^{-i\frac{\pi}{4}}\sqrt{\frac{k}{2\pi}}\right)^{n-2}
\frac{e^{ikr}}{r^{\frac{n-1}{2}}}\left\{u^{\infty}(\hx,\theta,k)
+\mathcal{O}\left(\frac{1}{r}\right)\right\}\quad\mbox{as }\,r:=|x|\rightarrow\infty,
\en
uniformly with respect to all directions $\hx:=x/|x|\in\Sp^{n-1}$.
The complex valued function $u^{\infty}=u^{\infty}(\hx,\theta,k)$ defined on $\Sp^{n-1}$
is known as the far field pattern with $\hx\in\Sp^{n-1}$ denoting the observation direction.

The inverse problem is concerned with the determination of the underlying objects with the far field patterns $u^{\infty}(\hx,\hth, k)$, $\hx,\hth\in \Sp^{n-1}$,
$k\in (k_{-}, k^{+})$. For some positive integer $l\in\Z$, define
\ben
\Theta_l:=\{ \theta_1,  \theta_2, \cdots, \theta_l |\,\theta_j\in \Sp^{n-1}, j=1,2,\cdots,l\},
\enn
which is a subset of $\Sp^{n-1}$ with finitely many directions.
In this paper, we consider the following three types of direction sets:
\begin{itemize}
  \item $\mathcal {A}_1:=\{(\hx,\hth)\in \Sp^{n-1}\times\Sp^{n-1} \,| \, \hx=-\hth,\,\forall \hth\in\Theta_l \}$;\\
        The first data set $\mathcal{A}_1$ is corresponding to the classical backscattering experiment, where the sensor plays the role of both the source and the receiver. In this setting, the measurements are taken by moving the sensor around the objects.
  \item $\mathcal {A}_2:=\{(\hx,\hth)\in \Sp^{n-1}\times\Sp^{n-1} \,| \, \hx=Q\hth,\,\forall \hth\in\Theta_l \}$;\\
        The second data set $\mathcal{A}_2$ is corresponding to a generalized backscattering experiment, where there are two sensors play the role of the source and the receiver, respectively. The measurements are then taken by moving these two sensors around the objects simultaneously.
  \item $\mathcal {A}_3:=\{(\hx,\hth)\in \Sp^{n-1}\times\Sp^{n-1} \,| \, \forall \hx\in\Theta_l, \,\mbox{and fixed } \hth \}$.\\
        The third data set $\mathcal{A}_3$ is corresponding to the other generalized backscattering experiment, where a fixed sensor plays the role of the source, and the other sensor plays the role of the receiver. The measurements are then taken by moving the receiver sensor around the objects.
\end{itemize}
Here, $Q$ is some fixed rotation in $\R^n$. In particular, the data set $\mathcal {A}_2$ reduces to the data set $\mathcal {A}_1$ if $Q=-I$, where $I$ is the identity matrix.

{\bf\em The inverse backscattering problem consists in the determination of the underlying object
from the far field patterns $u^{\infty}(\hx,\theta,k)$ for multiple frequencies $k\in (k_{-}, k^{+})$ at sparse pairs of directions $(\hx,\hth)\in \mathcal {A}_j$ with some $j\in \{1,2,3\}$.}

\section{Approximation uniqueness}

There are no uniqueness results for the inverse backscattering problems, in particular if the far field data is available at sparse directions. We still refer to \cite{HahnerKress,KR99} for some progress with full far field patterns are given as data. In this section, we investigate what kind of the information of the underlying object can be obtained by a knowledge of its multi-frequency far field patterns at one or two pairs of observation and incident directions.

We begin with the simple model with point like scatterers. From the asymptotic behavior of $\Phi_k(x,y)$ we deduce that the corresponding far field pattern is approximately given by \cite{KirschGrinberg}
\be\label{uinfzm}
u^{\infty}(\hx,\hth,k)\approx \wi{u}^{\infty}(\hx,\hth,k) := \sum_{m=1}^{M}\tau_m e^{ikz_m\cdot(\hth-\hx)},
\quad \hx, \hth\in \Sp^{n-1},\,k\in\{k_{-}, k^{+}\}.
\en

\begin{figure}[htbp]
\centering
\includegraphics[width=3in]{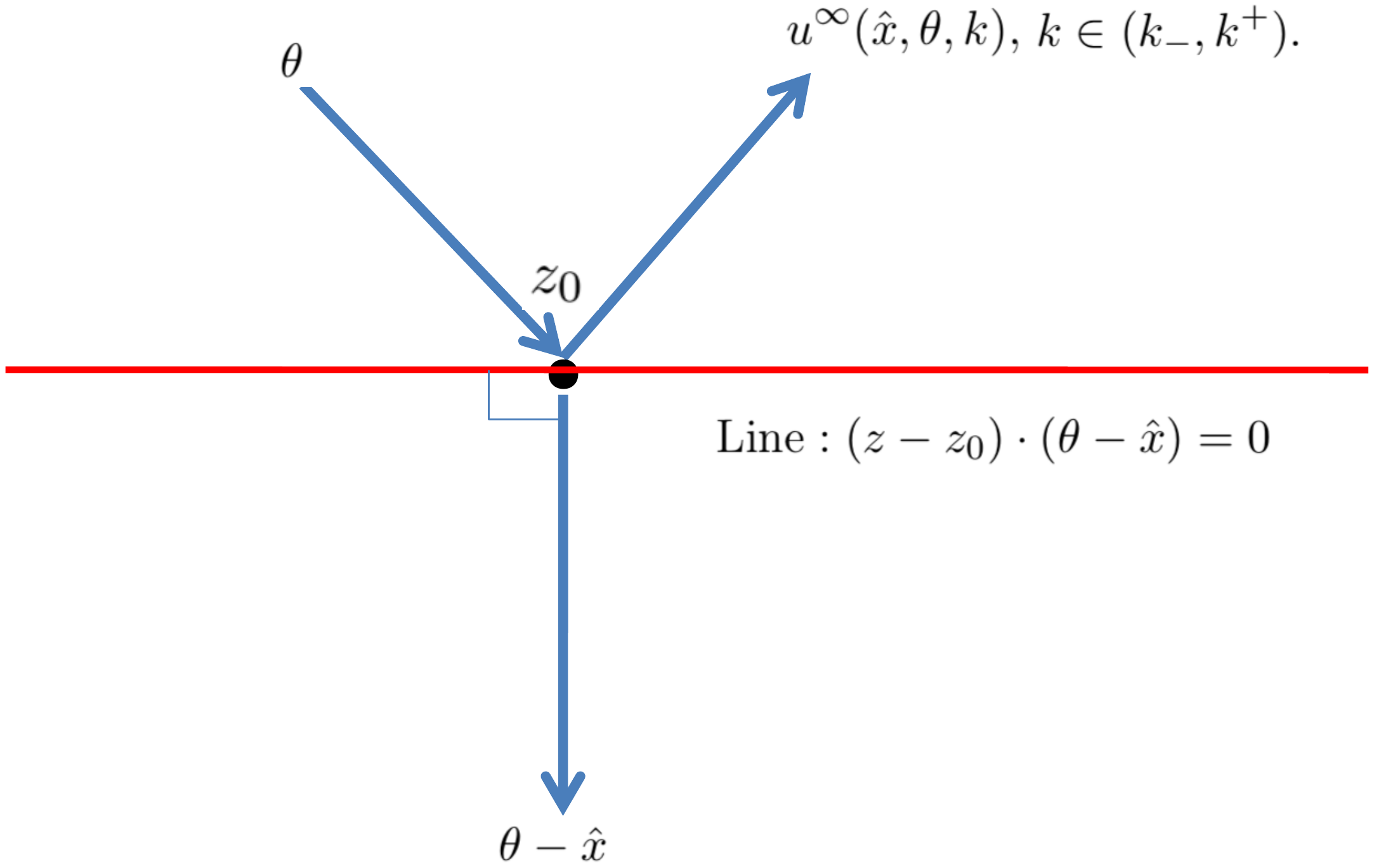}
\caption{Inverse scattering with the multi-frequency far field pattern $u^{\infty}(\hx,\hth,k)$ for fixed $\hx, \theta$.}
\label{pointline}
\end{figure}

As shown in Figure \ref{pointline}, for any fixed $(\hx,\theta)\in\mathcal {A}_j, \,j=1, 2, 3$ satisfying $\hx\neq\hth$,
define
\be\label{phi}
\phi:=\frac{\theta-\hx}{|\theta-\hx|}.
\en
Let
\be\label{Pialpha}
\Pi_{\alpha}:=\{y\in \R^n | y\cdot (\theta-\hx)+\alpha=0\}
\en
be the hyperplane with normal $\phi$ defined by \eqref{phi}. 

\begin{theorem}\label{LineUnique}
Let $(\hx,\theta)\in\mathcal {A}_j, \,j=1, 2, 3$ be fixed satisfying $\hx\neq\hth$.
The hyperplanes $\Pi_{\alpha_m}$, $\alpha_m:=-z_m\cdot (\theta-\hx)$, $m=1,2,\cdots,M$
are uniquely determined by the data $\wi{u}^{\infty}(\hx,\hth,k)$ for all $k\in(k_{-}, k^{+})$.
\end{theorem}
\begin{proof}
Note that the data $\wi{u}^{\infty}(\hx,\hth,k)$ is an analytic function with respect to the wave number, thus we have $\wi{u}^{\infty}(\hx,\hth,k)$ for all frequencies in $\R^{+}$.
Define
\be\label{Ip1}
I_{p}(z, \hx, \theta):= \int_{\R^{+}} \wi{u}^{\infty}(\hx,\hth,k) e^{-ikz\cdot(\theta-\hx)}dk, \quad z\in\R^n.
\en
Inserting \eqref{uinfzm} into \eqref{Ip1}, we find that
\ben
I_{p}(z, \hx, \theta) = \sum_{m=1}^{M}\tau_m\int_{\R^{+}} e^{ik(z_m-z)\cdot(\theta-\hx)}dk = \sum_{m=1}^{M}\tau_m\delta [(z_m-z)\cdot(\theta-\hx)], \quad z\in\R^n,
\enn
where $\delta$ is the Dirac delta function.
The function $I_{p}(z, \hx, \theta)$ is zero everywhere except for the points located on the hyperplanes $\G_m, m=1,2,\cdots, M$.
Thus the hyperplanes $\Pi_{\alpha_m}$ are uniquely determined by the function $I_{p}(z, \hx, \theta)$ and thus also by the data $\wi{u}^{\infty}(\hx,\hth,k)$ for all $k\in (k_{-}, k^{+})$ at fixed pair of directions $(\hx,\hth)\in \Sp^{n-1}$.
\end{proof}

Note that due to the model with point like scatterers is still not completely well established, we have the above uniqueness with the approximate far field data
$\wi{u}^{\infty}(\hx,\hth,k)$.
Such an approximation uniqueness result only needs a single pair of observation and incident directions. With more directions be considered, the positions of the point like scatterers can be approximately determined.
We also remark that, based on the MUSIC algorithm, following the arguments given in \cite{GriesmaierSchmiedecke}, one can show that
the positions $z_m, m=1,2,\cdots, M$ can be uniquely determined by the backscattering far field patterns for finitely many frequencies.

Considering now the scattering by inhomogeneous medium, it is well known that the total field $u=u^{in}+u^{s}$ satisfies the Lippmann-Schwinger equation \cite{CK}
\ben
u(x,\hth,k)=e^{ikx\cdot\hth}+k^2\int_{\R^n}\Phi_k(x,y)[q(y)-1]u(y,\hth,k)dy, \quad x\in\R^n.
\enn
Under the condition that the value of $k^2\max_{x\in D}\int_{\R^n}|\Phi_k(x,y)[q(y)-1]|dy$ is small enouth, the Lippmann-Schwinger equation can be solved iteratively by the Neumann series, and the first iteration is the Born approximation given by
\ben
u_{B}(x,\hth,k)=e^{ikx\cdot\hth}+k^2\int_{\R^n}\Phi_k(x,y)[q(y)-1]u^{in}(y,\hth,k)dy, \quad x\in\R^n.
\enn
Then the corresponding far field pattern is given by
\be\label{Born}
u_{B}^{\infty}(\hx,\hth,k)&=&k^2\int_{\R^n}e^{ik(\hth-\hx)\cdot y}[q(y)-1]dy\cr
&=&k^2\int_{\R}\int_{\Pi_\alpha}[q(y)-1] e^{ik(\theta-\hx)\cdot y}ds(y)d\alpha \cr
&=&k^2\int_{\R}\hat{q}(\alpha) e^{-ik\alpha}d\alpha, \quad k\in (k_{-}, k^{+}),
\en
where $\Pi_\alpha$ is defined by \eqref{Pialpha} and
\be
\hat{q}(\alpha):=\int_{\Pi_\alpha}[q(y)-1]ds(y).
\en

As shown in Figure \ref{IlluminatedShadow}, for any fixed direction $\theta\in\Theta_l$, the $\theta$-strip hull of $D$ is defined by
\ben
S_{D}(\theta):=  \{ y\in \mathbb \R^{n}\; | \; \inf_{z\in D}z\cdot \theta \leq y\cdot \theta\leq \sup_{z\in D}z\cdot \theta\},
\enn
which is the smallest strip (region between two parallel hyper-planes) with normals in the directions $\pm \theta$ that contains $\ov{D}$.

\begin{figure}[htbp]
\centering
\includegraphics[width=3in]{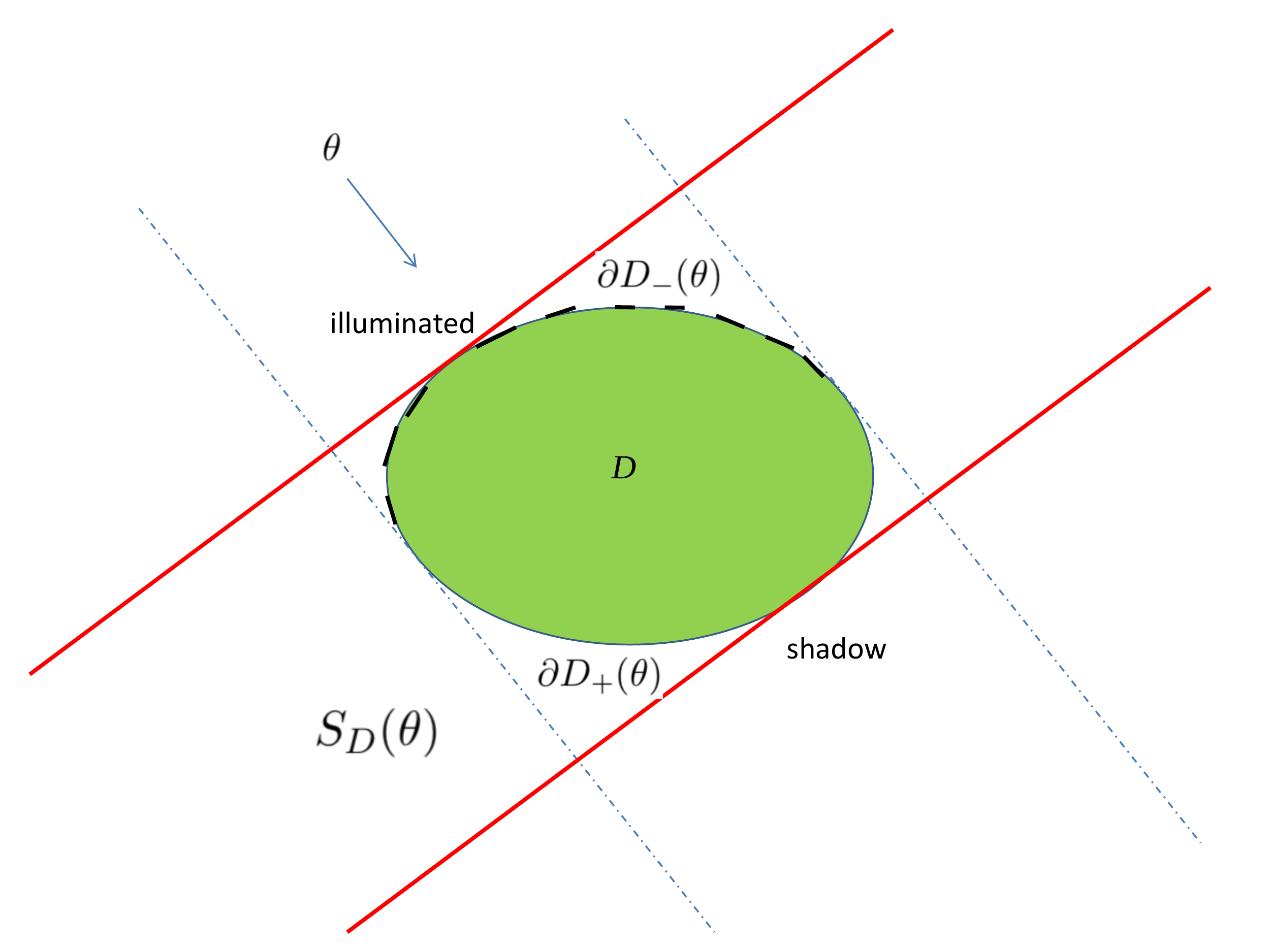}
\caption{The $\theta$-strip $S_{D}(\theta)$, illuminated part $\pa D_{-}(\theta)$ and shadow region $\pa D_{+}(\theta)$.}
\label{IlluminatedShadow}
\end{figure}

Following the arguments in the proof of Theorem 2.2 in \cite{AlaHuLiuSun}, we have the following uniqueness result for the inverse medium scattering problem.

\begin{theorem}\label{Strip-uniqueness-medium}
For any fixed $(\hx,\theta)\in\mathcal {A}_j, \,j=1, 2, 3$ satisfying $\hx\neq\hth$, let $\phi$ be defined as in \eqref{phi}.
If the set
\be\label{setmedium}
\{\alpha\in\R|\, \Pi_\alpha\subset S_D(\phi), \hat{q}(\alpha)=0\}
\en
has Lebesgue measure zero,
then the corresponding $\phi$-strip $S_D(\phi)$ is uniquely determined by the data $u_{B}^{\infty}(\hx,\hth,k)$ for all $k\in (k_{-}, k^{+})$.
\end{theorem}

Finally, we consider the case with extended obstacles. Difficulty also arises due to the nonlinearity from the obstacle to the far field pattern.
To solve this problem, we recall the other linearized method based on the Kirchhoff or physical optics approximation.

For a convex obstacle $D$, let
\be\label{illuminatedShadow}
\pa D_{-}(\theta):=\{x\in\pa D |\, \nu(x)\cdot \theta<0\} \quad\mbox{and}\quad \pa D_{+}(\theta):=\{x\in\pa D |\, \nu(x)\cdot \theta\geq0\}
\en
be the illuminated region and shadow region, respectively, with respect to the plane wave in the incident direction $\theta$.
For large wave number $k$, i.e., for small wavelengths, an obstacle $D$ locally may be considered at each point $x\in\pa D$ as a hyperplane
with normal $\nu(x)$. This leads to setting \cite{CK,LiLiu}
\be\label{Kirschhoff-soft}
\frac{\pa u}{\pa\nu} \approx \left\{
                            \begin{array}{ll}
                              2\frac{\pa u^{in}}{\pa\nu} , & \hbox{on $\pa D_{-}(\theta)$;} \\
                              0, & \hbox{on $\pa D_{+}(\theta)$.}
                            \end{array}
                          \right.
\en
if $D$ is sound soft, and
\be\label{Kirschhoff-hard}
u \approx \left\{
            \begin{array}{ll}
              2 u^{in}, & \hbox{on $\pa D_{-}(\theta)$;} \\
              0, & \hbox{on $\pa D_{+}(\theta)$.}
            \end{array}
          \right.
\en
if $D$ is sound hard.
Let $k_{-}>0$ be large enough such that the Kirchhoff approximation \eqref{Kirschhoff-soft}-\eqref{Kirschhoff-hard} holds.

It is well known that the far field pattern has the following representation \cite{KirschGrinberg}
\be\label{Far-representation}
u^{\infty}(\hx,\theta,k)=\int_{\pa D}\left\{u(y,\theta)\frac{\pa e^{-ik\hx\cdot y}}{\pa\nu(y)}-\frac{\pa u}{\pa\nu}(y,\theta)e^{-ik\hx\cdot y}\right\}ds(y),\quad \hx,\theta\in \Sp^{n-1},\quad k>0.
\en
Inserting the Kirchhoff approximations \eqref{Kirschhoff-soft}-\eqref{Kirschhoff-hard} into the above far field representation \eqref{Far-representation}, with the
help of the boundary condition \eqref{Bc},
the far field pattern in the back scattering direction is approximately given by \cite{CK}
\be\label{Kirschhoff}
u^{\infty}(\hx,\theta,k)\approx\g_{D}\int_{\pa D_{-}(\theta)}\frac{\pa e^{ik\theta\cdot y}}{\pa\nu(y)} e^{-ik\hx\cdot y}ds(y),\quad \theta\in \Sp^{n-1}, \quad k>k_{-},
\en
where
\ben
\g_{D} := \left\{
            \begin{array}{ll}
              -2, & \hbox{if $D$ is sound soft;} \\
              2, & \hbox{if $D$ is sound hard.}
            \end{array}
          \right.
\enn
An interesting observation from \eqref{Kirschhoff} is that the shadow region $\pa D_{+}(\theta)$ gives no contribution to the far field data
$u^{\infty}(\hx,\theta,k)$. Thus, it is impossible to reconstruct the shadow region $\pa D_{+}(\theta)$ from the far field data
$u^{\infty}(\hx,-\theta,k)$, in particular for high frequency case.
Replacing $\hx$ and $\theta$ by $-\hx$ and $-\theta$, respectively, we have
\ben
u^{\infty}(-\hx,-\theta,k)\approx \g_{D}\int_{\pa D_{+}(\theta)} \frac{\pa e^{-ik\theta\cdot y}}{\pa\nu(y)} e^{ik\hx\cdot y}ds(y),\quad \theta\in \Sp^{n-1}, \quad k>k_{-}.
\enn
Combining the last two equations we find
\ben
u^{\infty}(\hx,\theta,k)+\ov{u^{\infty}(-\hx,-\theta,k)} \approx \g_{D}\int_{\pa D}\frac{\pa e^{ik\theta\cdot y}}{\pa\nu(y)} e^{-ik\hx\cdot y}ds(y),\quad \theta\in \Sp^{n-1}, \quad k>k_{-}.
\enn
Furthermore, by the Gauss divergence theorem, we deduce that
\be\label{GeneralizedBojarski}
&&u^{\infty}(\hx,\theta,k)+\ov{u^{\infty}(-\hx,-\theta,k)}\cr
&\approx& \g_{D}\int_{ D}\ddiv[e^{-ik\hx\cdot y}\nabla e^{ik\theta\cdot y}]   d(y)\cr
&=&-k^2\hth\cdot(\hth-\hx)\g_{D}\int_{\R^n}\chi(y) e^{ik(\theta-\hx)\cdot y} d(y)\cr
&=:& U^{\infty}(\hx,\theta,k),\quad \hx, \theta\in \Sp^{n-1}, \quad k>k_{-},
\en
where $\chi$ is the characteristic function of the domain $D$.
This is the well known {\em Bojarski identity} \cite{Bojarski} if $\hx=-\hth$. We call \eqref{GeneralizedBojarski} the {\em generalized Bojarski identity.}
The {\em Bojarski identity}
\cite{Bojarski} indicates that, in the Kirchhoff approximation, the Fourier transform of the characteristic function of the scatterer can be completely determined from
the backscattering far field patterns for all observation directions and all positive wave numbers. Then, by inverting the Fourier transform one can determine the
location
and shape of the scatterers. However, this procedure suffers from two difficulties. Firstly, the Kirchhoff approximation is valid for high frequencies only, whereas the
inverse Fourier transform requires integration over all frequencies. Secondly, the inverse Fourier transform requires all observation directions, whereas in most practical
situations, measurements are only available for incomplete set of observation directions. We refer to \cite{Bojarski, Devaney, Langenberg} for the numerical implementations
and more discussions.

The following theorem gives an uniqueness result based on the data $U^{\infty}(\hx, \theta,k)$, which is an approximation of $u^{\infty}(-\theta,\theta,k)+\ov{u^{\infty}(\theta,-\theta,k)}$ for all $k\in (k_{-}, k^{+})$.

\begin{theorem}\label{Strip-uniqueness}
For any fixed $(\hx,\theta)\in\mathcal {A}_j, \,j=1, 2, 3$ satisfying $\hx\neq\hth$, let $\phi$ be defined as in \eqref{phi}.
Then the $\phi$-strip $S_D(\phi)$ is uniquely determined by the data $U^{\infty}(\hx, \theta,k)$ for all $k\in (k_{-}, k^{+})$.
\end{theorem}
\begin{proof}
Recall the {\em generalized Bojarski identity} \eqref{GeneralizedBojarski}, we have
\be\label{UinfFourier}
V^{\infty}(\hx,\theta,k)
&:=&\frac{U^{\infty}(\hx,\theta,k)}{-k^2\hth\cdot(\hth-\hx)\g_{D}}\cr
&=&\int_{\R^n}\chi(y) e^{ik(\theta-\hx)\cdot y} d(y)\cr
&=&\int_{\R}\int_{\Pi_\alpha}\chi(y) e^{ik(\theta-\hx)\cdot y}ds(y)d\alpha \cr
&=&\int_{\R}\hat{\chi}(\alpha) e^{-ik\alpha}d\alpha, \quad k\in (k_{-}, k^{+}),
\en
where
\be
\hat{\chi}(\alpha):=\int_{\Pi_\alpha}\chi(y)ds(y).
\en
Note that $V^{\infty}(\hx,\theta,k)$ is an analytic function with respect to the wave number $k$. Thus we have the data $V^{\infty}(\hx,\theta,k)$ for all $k\in\R$ by analyticity. The equality \eqref{UinfFourier} implies that $V^{\infty}(\hx,\theta,k)$ is the Fourier transform of $\hat{\chi}$. Using inverse Fourier transform, we deduce that $\hat{\chi}$ can be uniquely determined.
Note that
\ben
S_D(\phi)=\ov{\bigcup_{\alpha\in \R} \{\Pi_{\alpha}| \hat{\chi}(\alpha)\neq 0\}},
\enn
which implies that the strip $S_D(\phi)$ is uniquely determined by $\hat{\chi}$, and thus by the data $U^{\infty}(\hx,\theta,k)$ for all $k\in (k_{-}, k^{+})$ at a single pair of directions $(\hx,\theta)\in\mathcal {A}_j, \,j=1, 2, 3$. The proof is complete.
\end{proof}

Note that the approximation far field patterns given in \eqref{uinfzm}, \eqref{Born} and \eqref{GeneralizedBojarski}, respectively, for three different problems are linear with respect to the support of the characteristic function for the underlying object. This is the basic observation used for proof of the approximation uniqueness theorems \ref{LineUnique}-\ref{Strip-uniqueness}.
Although we can not give the rigorous and exact statement of uniqueness results, these approximation uniqueness theorems are enough to ensure the effectiveness of the subsequent numerical methods, in particular noting the fact that the measurement errors can not be avoided in the practical applications.

\section{Direct sampling method}

Motivated by the approximation far field representations given in \eqref{uinfzm}, \eqref{Born} and \eqref{GeneralizedBojarski}, with the three different generalized sparse data sets $\mathcal {A}_j, \,j=1,2,3$, we introduce the following two indicators
\be\label{I1}
I^{(j)}_1(z)&:=&\sum_{(\hx,\theta)\in\mathcal {A}_j}I_1(z,\hx,\theta)\cr
&:=&\sum_{(\hx,\theta)\in\mathcal {A}_j}\left|\int_{k_{-}}^{k^{+}}u^{\infty}(\hx,\theta,k) e^{-ikz\cdot(\theta-\hx)}dk\right|,\quad z\in\R^n,
\en
and
\be\label{I2}
I^{(j)}_2(z)&:=&\sum_{(\hx,\theta)\in\mathcal {A}_j}I_2(z,\hx,\theta)\cr
&:=&\sum_{(\hx,\theta)\in\mathcal {A}_j}\left|\int_{k_{-}}^{k^{+}}\Big(u^{\infty}(\hx,\theta,k)+\ov{u^{\infty}(-\hx,-\theta,k)}\Big) e^{-ikz\cdot(\theta-\hx)}dk\right|,
\quad z\in\R^n.
\en

Straightforward calculation shows that we have the following property of the indicators.

\begin{proposition}
For fixed $(\hx,\theta)\in\mathcal {A}_j, \,j=1,2,3$ with $\hx\neq\hth$, we have
\ben
I_l(z_1,\hx,\theta) = I_l(z_2,\hx,\theta) \quad \mbox{for}\, z_1,z_2\in\Pi_{\alpha},\, l=1,2,
\enn
where $\Pi_{\alpha}$ is the hyperplane given by \eqref{Pialpha} for some fixed $\alpha\in\R$.
\end{proposition}

This implies that when we consider the reconstructions using the indicator $I_l(z,\hx,\theta)$ with a single pair of directions, we just need to compute the values of $I_l(z,\hx,\theta)$ for sampling points in a line along the direction $\hth-\hx$.

To further explain why the two indicators works, we consider again the weak scattering (low frequency) approximation and the Kirchhoff (high frequency) approximation.
For frequencies in the resonance region, a rigorous proof is not known since the inverse problem is full nonlinear. However, the numerical experiments in the next section indicate that our direct sampling methods work very well. We focus on the first indicator $I_1(z,\hx,\theta)$ since the other indicator $I_2(z,\hx,\theta)$
has similar behaviors.

We first consider three types of weak scattering approximation: point like scatterers, small scatterers and low contrast inhomogeneities.

For point like scatterers, with the help of \eqref{uinfzm}, letting $k_{-}\rightarrow 0$ and $k^{+}\rightarrow\infty$, we obtain
\be\label{I1behaviorpoints}
I_1(z,\hx,\theta)
&\approx&\left|\sum_{m=1}^{M}\tau_m\int_{\R^{+}} e^{-ik(z-z_m)\cdot(\theta-\hx)}dk\right|\cr
&=&\left|\sum_{m=1}^{M}\tau_m \delta[(z-z_m)\cdot(\theta-\hx)]\right|,
\en
where $\delta$ is the Dirac delta function. This implies that the indicator $I_1$ takes local maximum for sampling points on the hyperplanes
$\Pi_{m}: (z-z_m)\cdot(\theta-\hx)=0, m=1,2,\cdots, M$.

For small scatterers, denote by $D=\cup_{m=1}^{M} D_m$, where $D_m:=z_m + \rho B_m$ with $z_m$ describes the position, and $B_m$ their corresponding shape. It is proved in \cite{Griesmaier} that
\ben
u^{\infty}(\hx,\hth,k)\approx \sum_{m=1}^{M}c_m e^{ikz_m\cdot(\theta-\hx)}, \quad \mbox{as}\,\rho\rightarrow 0,
\enn
where $c_m$ depends on $k, \rho$ and $B_m$, but independent on the location $z_m, m=1,2,\cdots,M$.
Clearly, the coefficients $c_m$ play the role of the scattering strengths $\tau_m$ in \eqref{uinfzm}. Thus, we could expect that the indicator takes the same behavior as the one for the point like scatterers. To summarize, the indicator has the following property.

\begin{proposition}
Let $(\hx,\theta)\in\mathcal {A}_j, \,j=1,2,3,$ be fixed with $\hx\neq\hth$.
For point like scatterers and small scatterers, the indicator $I_1$ takes local maximum for sampling points on the hyperplanes
$\Pi_{m}: (z-z_m)\cdot(\theta-\hx)=0, m=1,2,\cdots, M$.
\end{proposition}

Inserting \eqref{Born} into $I_1(z,\hx, \theta)$ and interchanging the order of integration, we have
\be\label{I1behavior}
I_1(z,\hx,\theta)
&=&\left|\int_{k_{-}}^{k^{+}}k^2\int_{\R^n}e^{ik(\hth-\hx)\cdot y}[q(y)-1]dy e^{-ikz\cdot(\theta-\hx)}dk\right|\cr
&=&\left|\int_{\R^n}[q(y)-1]\int_{k_{-}}^{k^{+}}k^2 e^{-ik(z-y)\cdot(\theta-\hx)}dk dy\right|\cr
&=&\left|\int_{\R^n}[q(y)-1]\frac{f(y,z,\hx,\theta)}{(z-y)\cdot(\theta-\hx)}dy\right|,\quad z\in\R^n, \theta\in\Theta_l
\en
with
\ben
f(y,z,\hx,\theta):=\Big[k^2-\frac{2k}{i(z-y)\cdot(\theta-\hx)}-\frac{2}{|(z-y)\cdot(\theta-\hx)|^2}\Big] e^{-ik(z-y)\cdot(\theta-\hx)}\Big|_{k_{-}}^{k^{+}}.
\enn

Note that $q-1$ is compactly supported, therefore 
\ben
|f(y,z,\hx,\theta)|= O(1)\quad \mbox{as}\quad dist(z, S_D(\phi))\rightarrow\infty,
\enn
where $dist(z, S_D(\phi))$ is the distance between the sampling point $z$ and the strip $S_D(\phi)$.
Consequently we deduce from \eqref{I1behavior} the following asymptotic behavior.

\begin{proposition}\label{Proposition3}
Let $(\hx,\theta)\in\mathcal {A}_j, \,j=1,2,3,$ be fixed with $\hx\neq\hth$. Then for scattering by inhomogeneous medium,
\ben
I_1(z,\hx,\theta) = O\left(\frac{1}{dist(z, S_D(\phi))}\right) \quad \mbox{as}\quad dist(z, S_D(\phi))\rightarrow\infty.
\enn
\end{proposition}


We now consider the scattering by extended obstacles based on Kirchhoff approximation. Given a pair of directions $(\hx,\hth)\in \mathcal {A}_j, \,j=1,2,3,$ with $\hx\neq \hth$, we define the reflection coefficient
\ben
R^{\g}(\hx,\hth):=\g\kappa(y^{+})^{-1/2}|\hth-\hx|^{(1-n)/2}\phi\cdot\hx,
\enn
where $\phi$ is given by \eqref{phi}, $y^{+}\in \pa D$ is the preimage of $\phi$ under the Gauss map, $\kappa(y^{+})$ is the Gauss curvature at $y^{+}$, $\g=-1$ for the Dirichlet problem and $\g=1$ for the Neumann problem.
For a strictly convex obstacle $D$, it is proved that \cite{Majda}
\ben
u^{\infty}(\hx,\hth, k) =  R^{\g}(\hx,\hth) e^{ik y^{+}\cdot(\hth-\hx)} + O(k^{-1}),\quad k\geq k_{-}.
\enn
Inserting this into $I_1$, we obtain
\ben
I_1(z, \hx,\hth)
&\approx& |R^{\g}(\hx,\hth)|\int_{k_{-}}^{k^{+}} e^{ik (y^{+}-z)\cdot(\hth-\hx)} dk\\
&=& \frac{|R^{\g}(\hx,\hth)|}{|(y^{+}-z)\cdot(\hth-\hx)|} e^{ik (y^{+}-z)\cdot(\hth-\hx)}\Big|_{k_{-}}^{k^{+}}.
\enn

From this, we expect the following behavior of the indicator $I_1$ for extended convex obstacles.

\begin{proposition}\label{Proposition4}
Let $(\hx,\theta)\in\mathcal {A}_j, \,j=1,2,3,$ be fixed with $\hx\neq\hth$. Then for scattering by extended convex obstacles,
the indicator $I_1$ decays like
\ben
\frac{1}{|(y^{+}-z)\cdot(\hth-\hx)|}
\enn
when the sampling point $z$ moves away from the hyperplane $\Pi: (z-y^{+})\cdot(\hth-\hx)=0$.
\end{proposition}

Finally, we want to remark that the weak scattering approximation and the Kirchhoff approximation are only needed for justification of the proposed direct sampling method theoretically. However, numerical examples in the next section show that the direct sampling methods work very well, even if the obstacle has concave part or contains multiple multiscalar components.

\section{Numerical examples and discussions}
\label{NumExamples}
\setcounter{equation}{0}
In this section, a variety of numerical examples are presented in two dimensions to illustrate the applicability
and effectiveness of our sampling methods.
The boundaries of the scatterers used in our numerical experiments are parameterized as follows
\be
\label{kite}&\mbox{\rm Kite:}&\quad x(t)\ =(a,b)+\ (\cos t+0.65\cos 2t-0.65, 1.5\sin t),\quad 0\leq t\leq2\pi,\\
\label{circle}&\mbox{\rm Circle:}&\quad x(t)\ =(a,b)+r\, (\cos t, \sin t),\quad 0\leq t\leq2\pi,\quad
\en
with $(a,b)$ be the location of the scatterer which may be different in different examples
and $r$ be the radius of the circle.

In our simulations, if not stated otherwise, we will always consider $20$ equally distributed wave numbers in the frequency band $[10,20]$ (so the wavelength is in $[0.314, 0.628]$). The far field patterns are obtained by using the boundary integral equation method.
We further perturb these synthetic data with $10\%$ relative random noise.
With these perturbed data, we solve the inverse problems using indicators proposed in the previous section with $0.1$ as the sampling space. We choose \ben Q=
\left [
            \begin{array}{ll}
              0 & 1  \\
              1 & 0
            \end{array}
          \right]
\enn
in $\mathcal A_2$, i.e.,  $\hat x$ is obtained by rotating $\theta$ anti-clockwise with $\pi/2$.

\textbf{Example-1:} We start with the well known bench example with a sound soft kite. In this example, we compare the behaviors of our indicators using one incident direction $\theta=[1,0]$. Fig. \ref{Ex1} gives the reconstructions of kite shaped domain using $I_1^{(1)}$ with $\hx=[-1,0]$ and $I_1^{(2)}$ with $\hx= [0,1]$, respectively.  We obtain a highlighted line in both figures, the reconstructions are consistent with the results discussed in the previous sections.  \\
\begin{figure}[htbp]
  \centering
  \subfigure[\textbf{$I_1^{(1)}$.}]{
    \includegraphics[height=2in,width=2in]{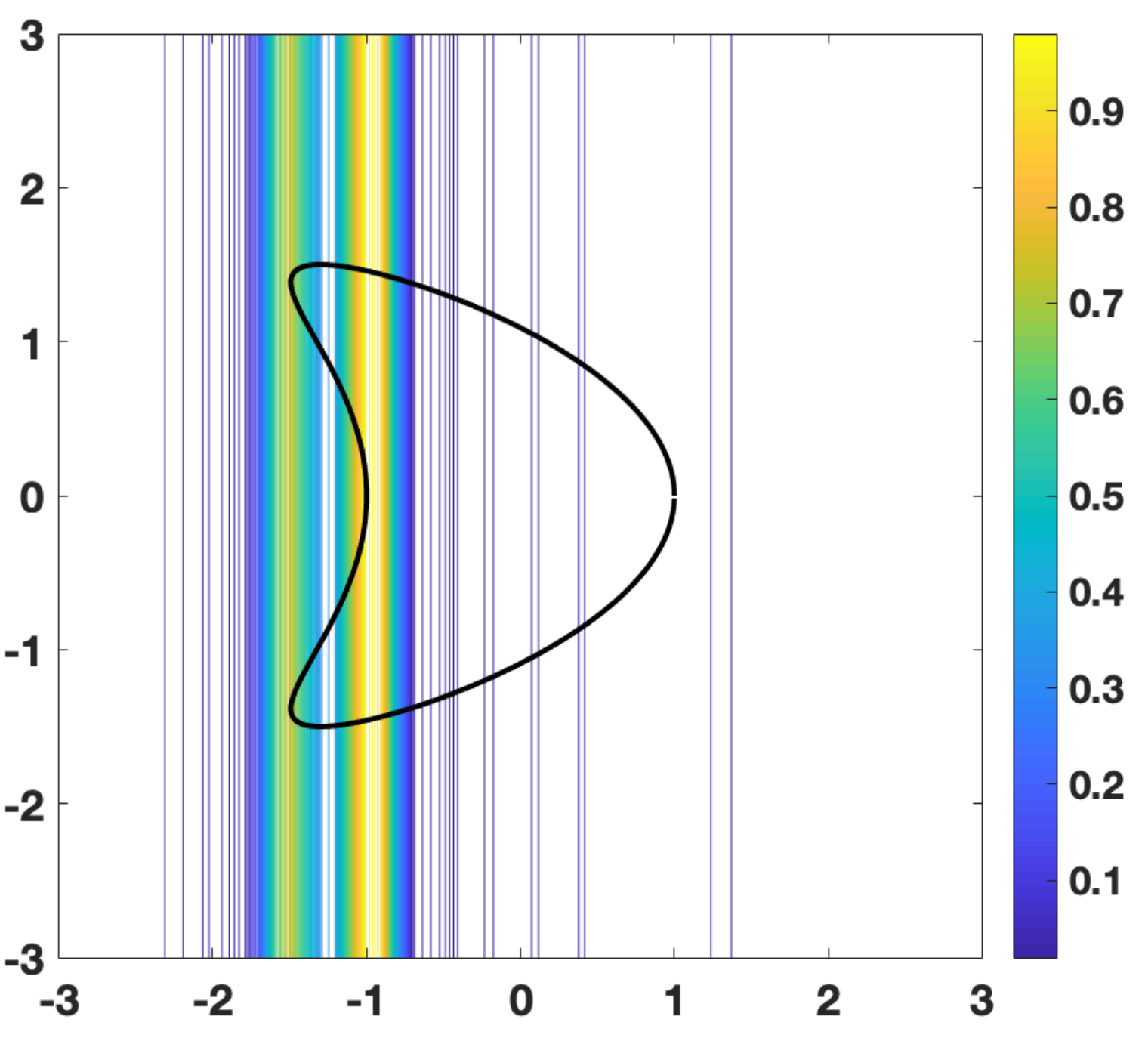}}
  \subfigure[\textbf{$I_1^{(2)}$.}]{
    \includegraphics[height=2in,width=2in]{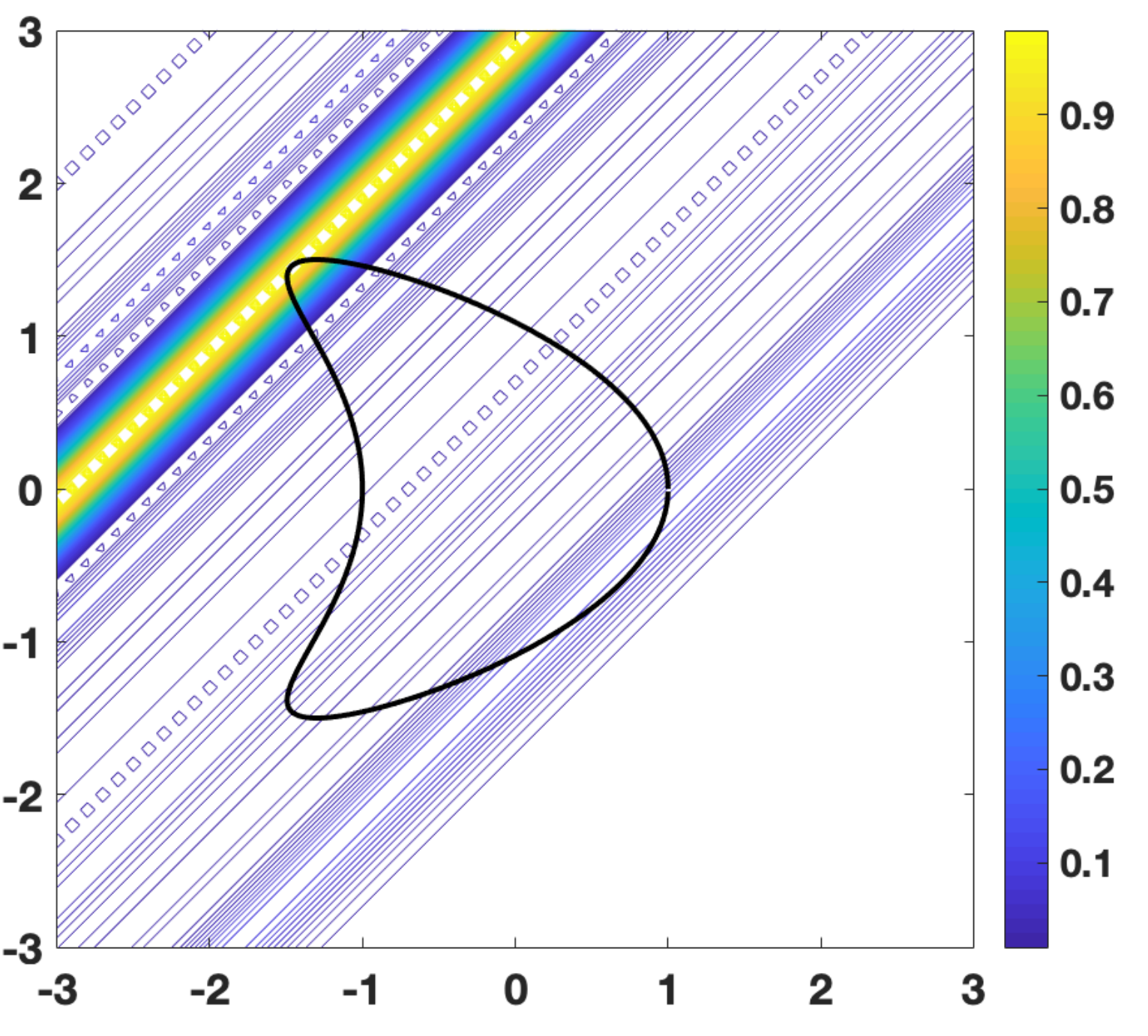}}
\caption{{\bf Example-1.}\, Reconstructions by $I_1^{(1)}$ and $I_1^{(2)}$ with a single pair of observation and incident directions.}\label{Ex1}
\end{figure}

\textbf{Example-2:}  The example has the same setup with \textbf{Example-1}, while we use two incident directions $\theta=[1,0],[-1,0]$. Fig. \ref{Ex2} gives the reconstructions of kite shaped domain with two incident directions using $I_1^{(1)}$ and $I_1^{(2)}$, respectively.  The smallest strip containing the kite shaped domain is well captured. 
\begin{figure}[htbp]
  \centering
  \subfigure[\textbf{$I_1^{(1)}$.}]{
    \includegraphics[height=2in,width=2in]{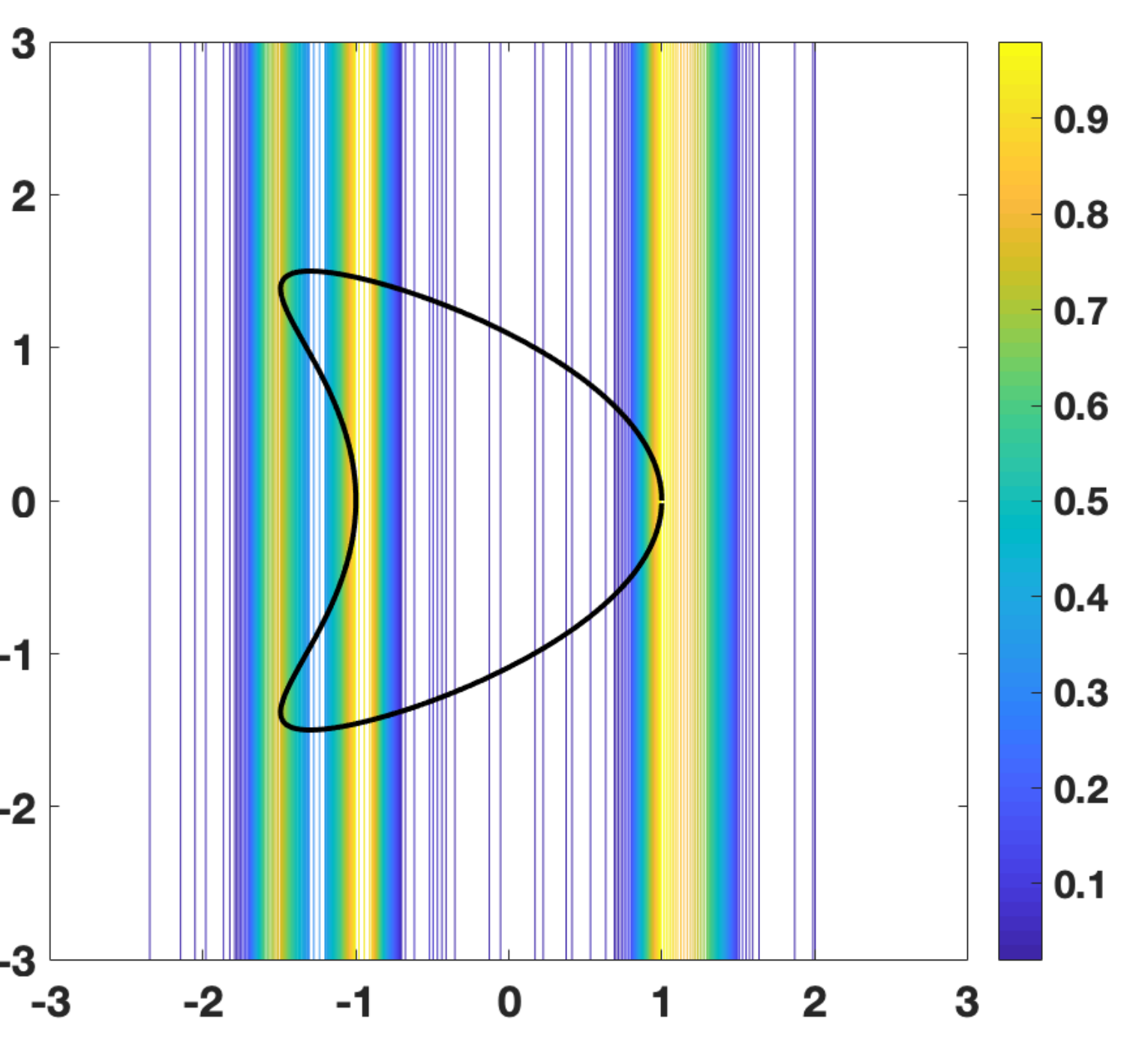}}
  \subfigure[\textbf{$I_1^{(2)}$.}]{
    \includegraphics[height=2in,width=2in]{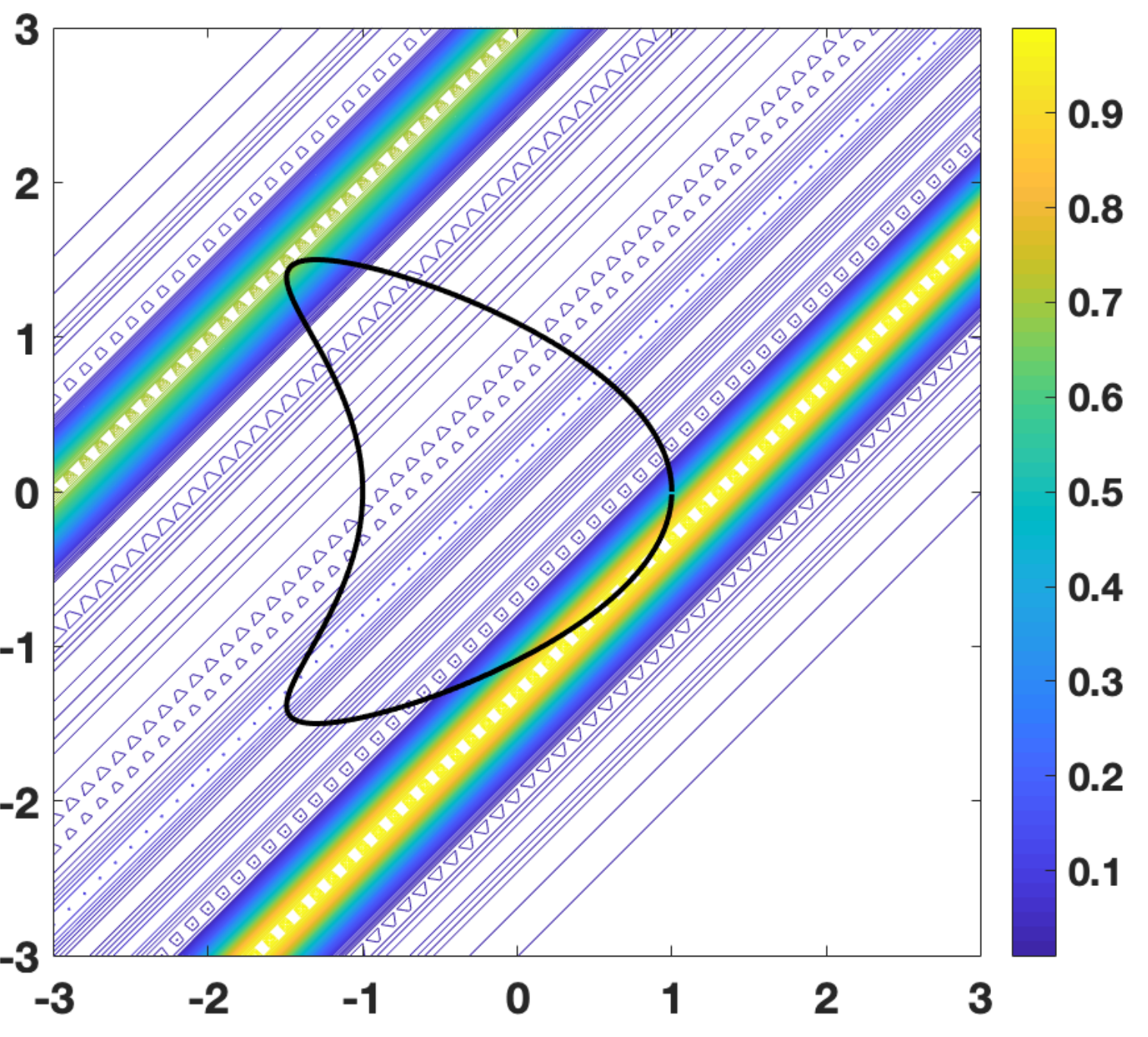}}
\caption{{\bf Example-2.}\, Reconstructions by $I_1^{(1)}$ and $I_1^{(2)}$ with two pairs of observation and incident directions.}\label{Ex2}
\end{figure}

 \textbf{Example-3:} In this example, we still consider the bench example with a sound soft kite.  We increase the number of incident directions to 32. Fig. \ref{Ex3} gives the reconstructions of kite shaped domain using $I_i^{(j)}, i,j=1,2$.  The shape of the kite is well constructed in all the figures. Surprisingly, the concave part is also clearly reconstructed. This example shows that our direct sampling method works well for non-convex scatterers.
\begin{figure}[htbp]
  \centering
  \subfigure[\textbf{$I_1^{(1)}$.}]{
    \includegraphics[height=2in,width=2in]{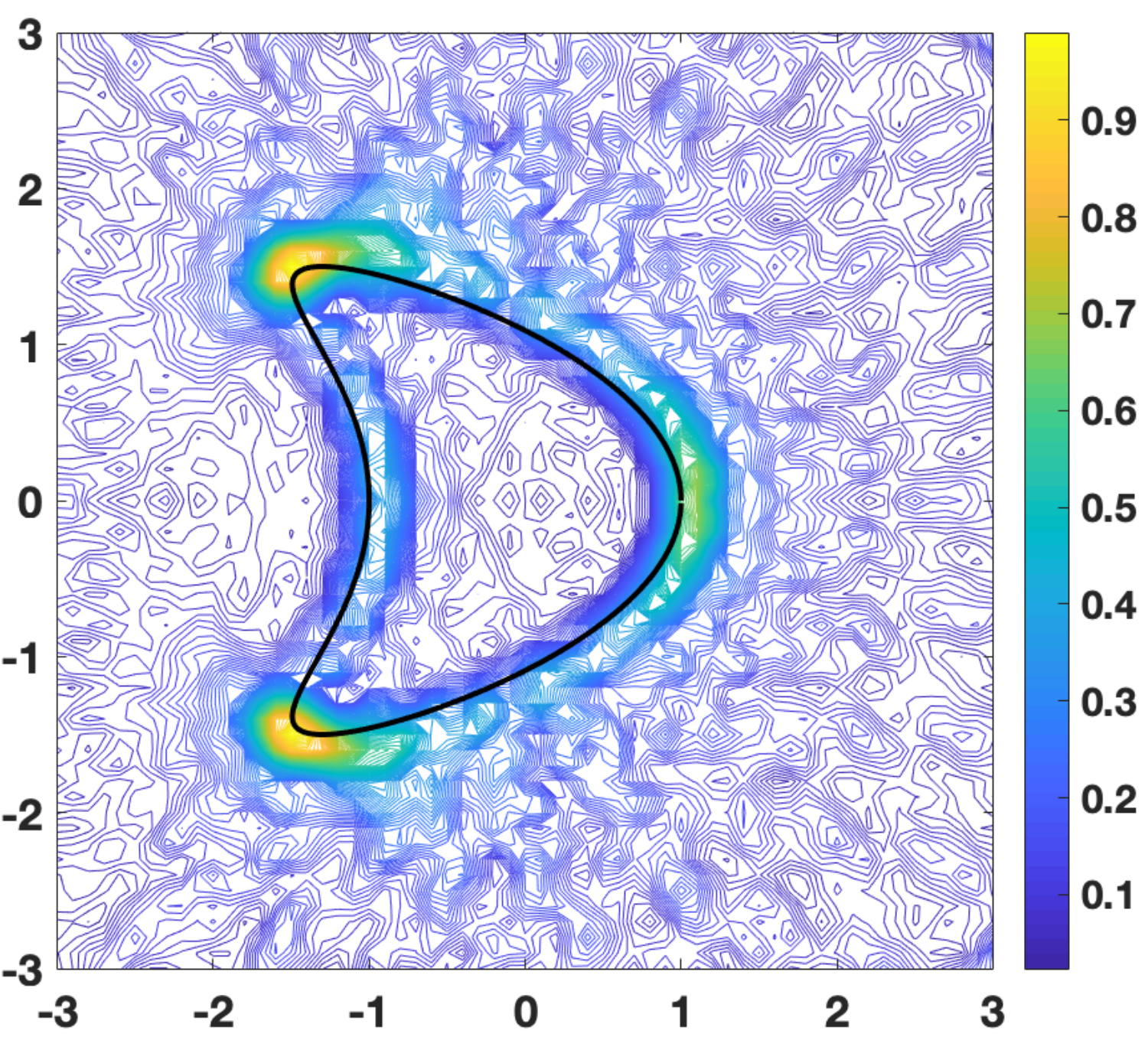}}
  \subfigure[\textbf{$I_1^{(2)}$.}]{
    \includegraphics[height=2in,width=2in]{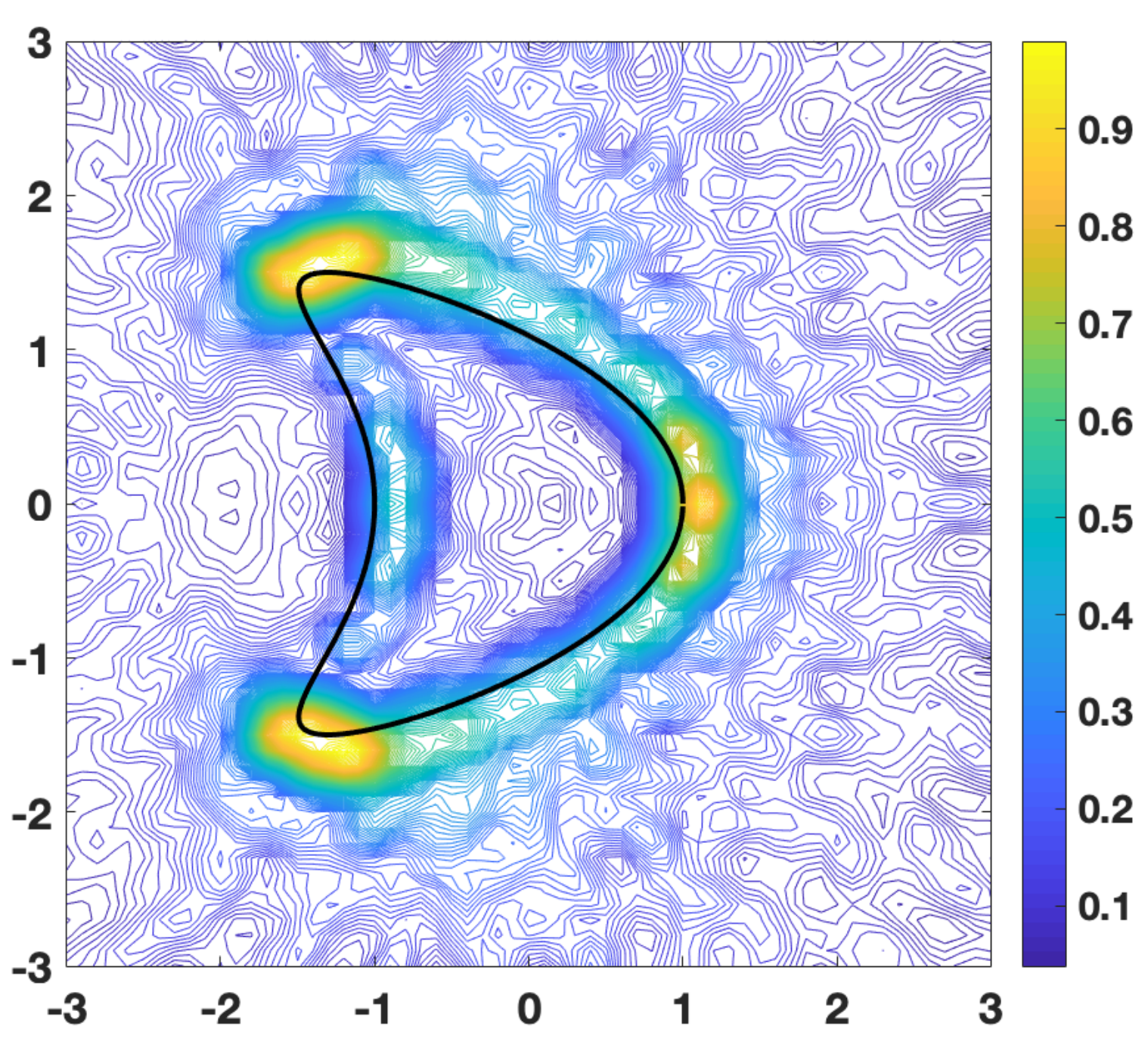}}\\
     \subfigure[\textbf{$I_2^{(1)}$.}]{
    \includegraphics[height=2in,width=2in]{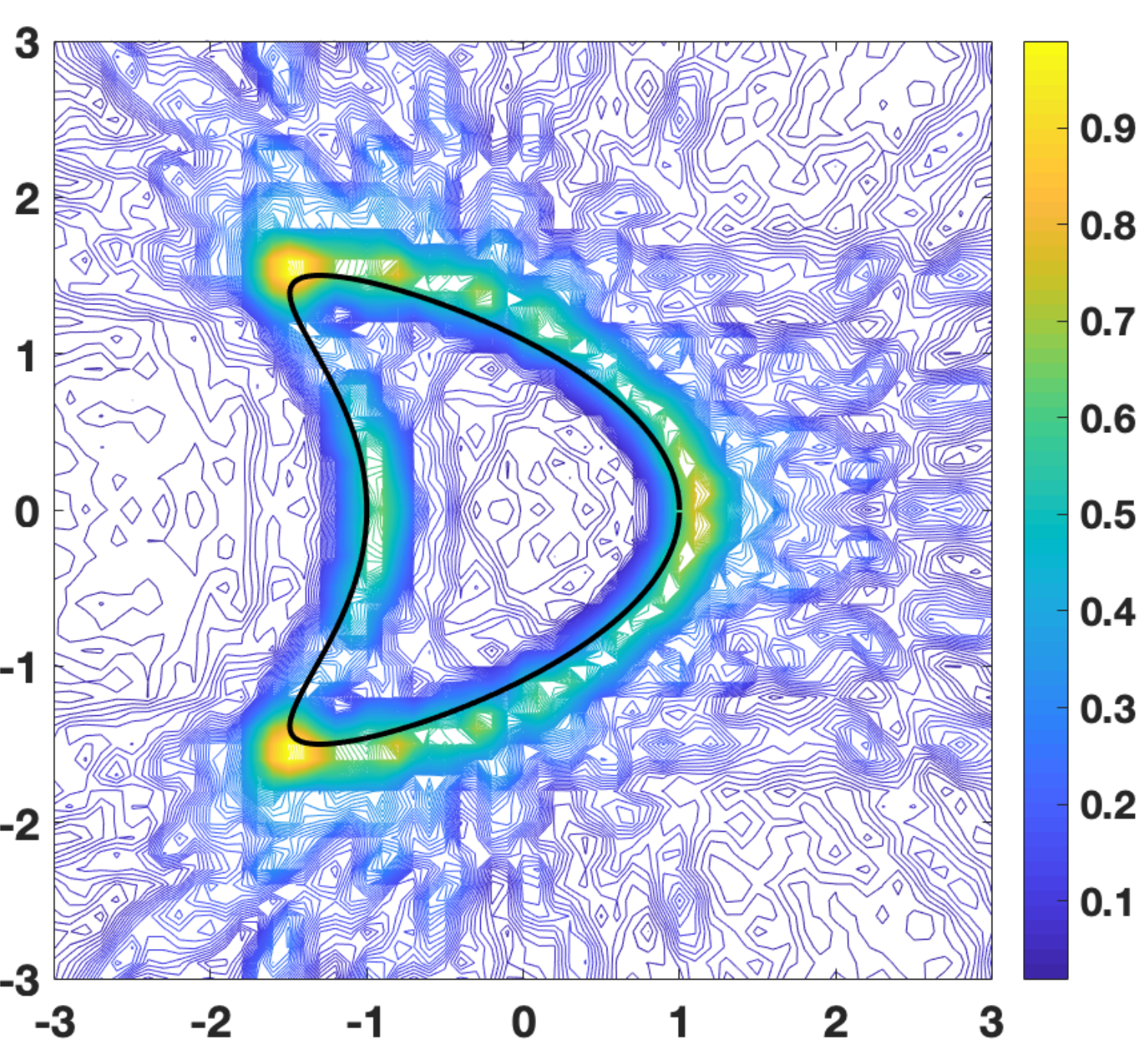}}
     \subfigure[\textbf{$I_2^{(2)}$.}]{
    \includegraphics[height=2in,width=2in]{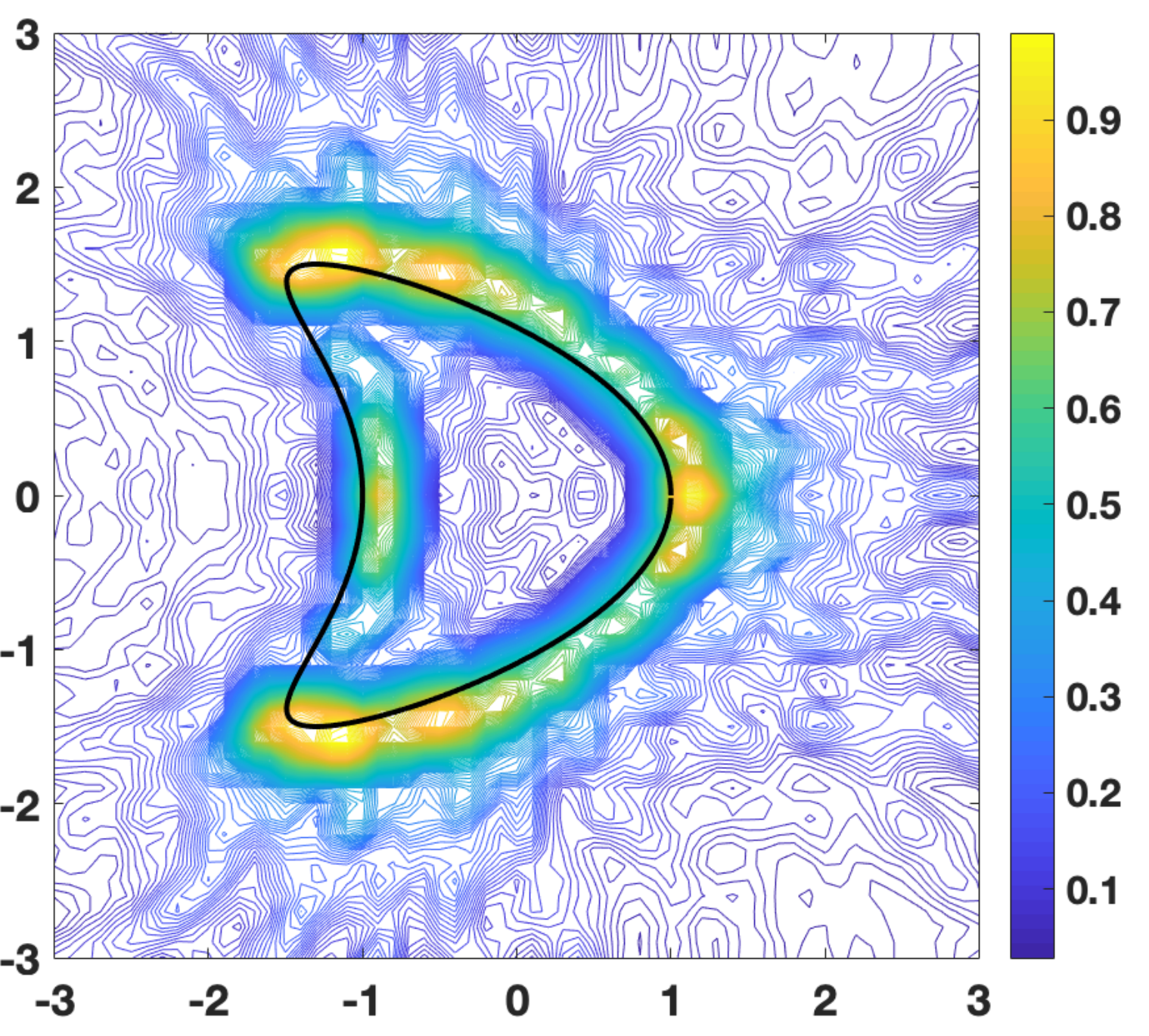}}
\caption{{\bf Example-3.}\, Reconstructions by by $I_1^{(j)}$ and $I_2^{(j)}$ with different data sets $\mathcal {A}_j, j=1,2$. $32$ pairs of directions are used.}\label{Ex3}
\end{figure}

We also consider the indicator $I^{(3)}_1$ with fixed incident direction $\theta$ and $32$ observation directions.  Fig. \ref{Ex32} gives the results. It's obvious that the constructions are good in the illuminated part.
\begin{figure}[htbp]
  \centering
  \subfigure[\textbf{$\theta=[1,0]$.}]{
    \includegraphics[height=2in,width=2in]{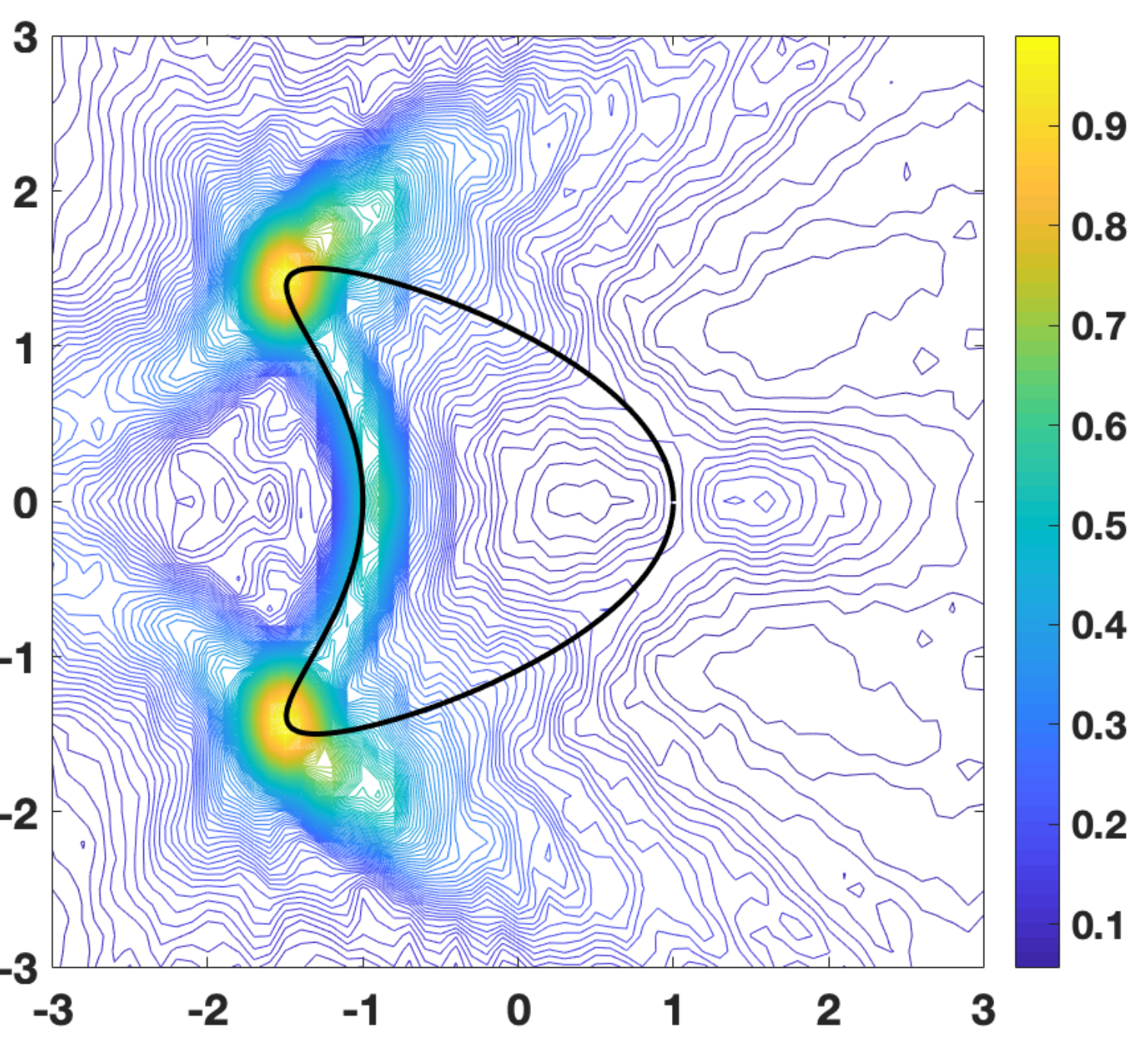}}
  \subfigure[\textbf{$\theta=[0,-1]$.}]{
    \includegraphics[height=2in,width=2in]{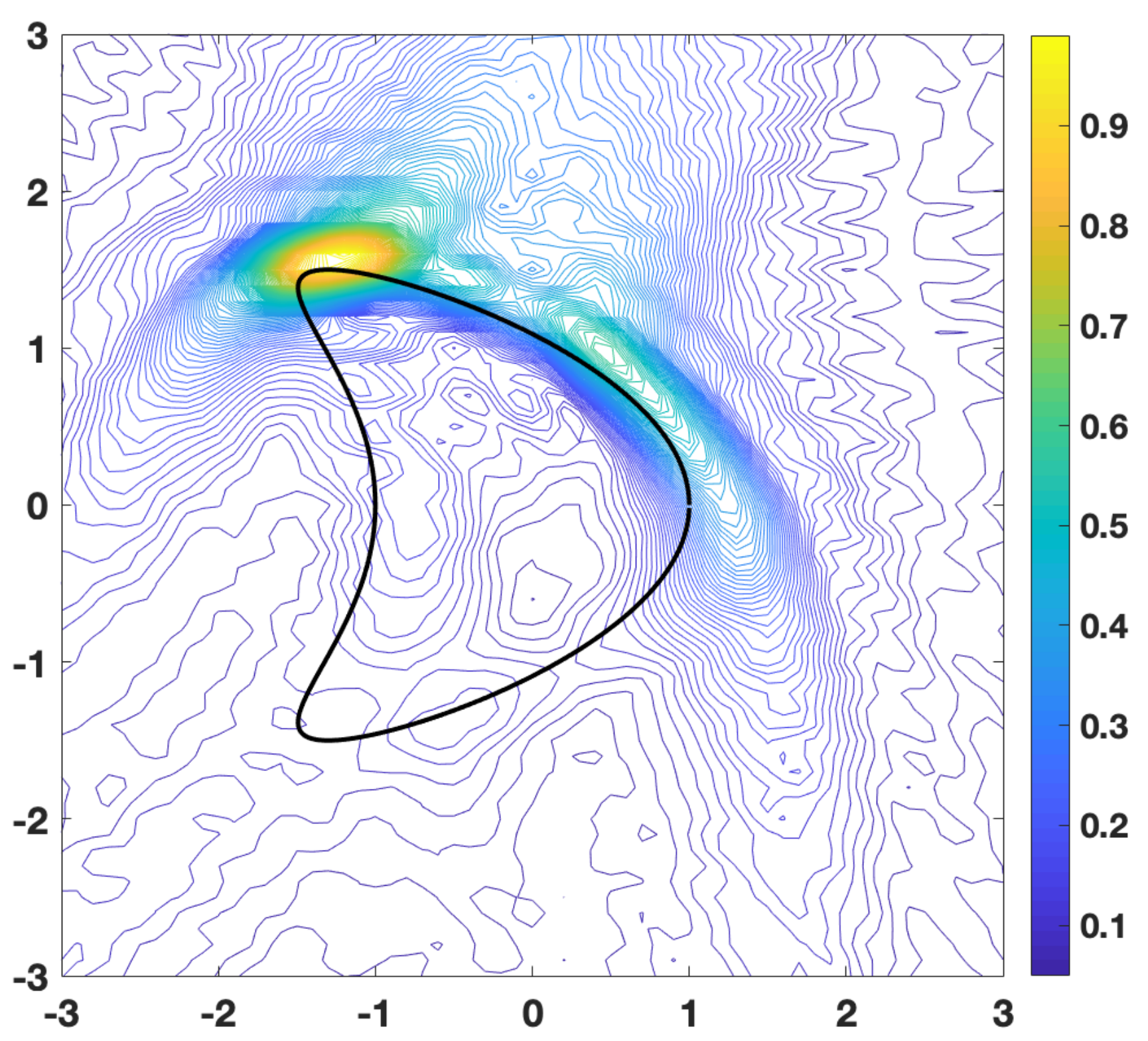}}\\
     \subfigure[\textbf{$\theta=[-1,0]$.}]{
    \includegraphics[height=2in,width=2in]{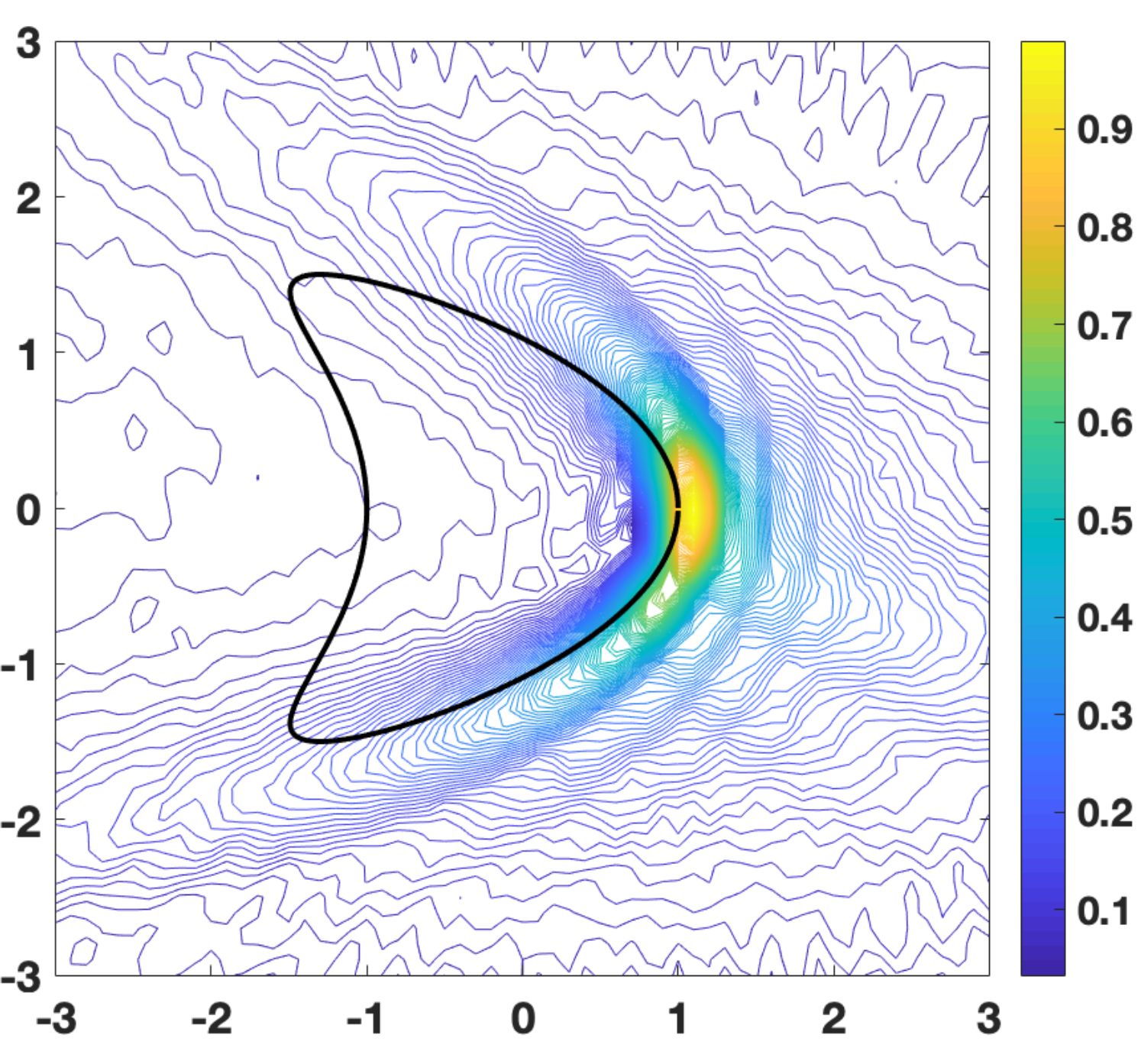}}
     \subfigure[\textbf{$\theta=[0,1]$.}]{
    \includegraphics[height=2in,width=2in]{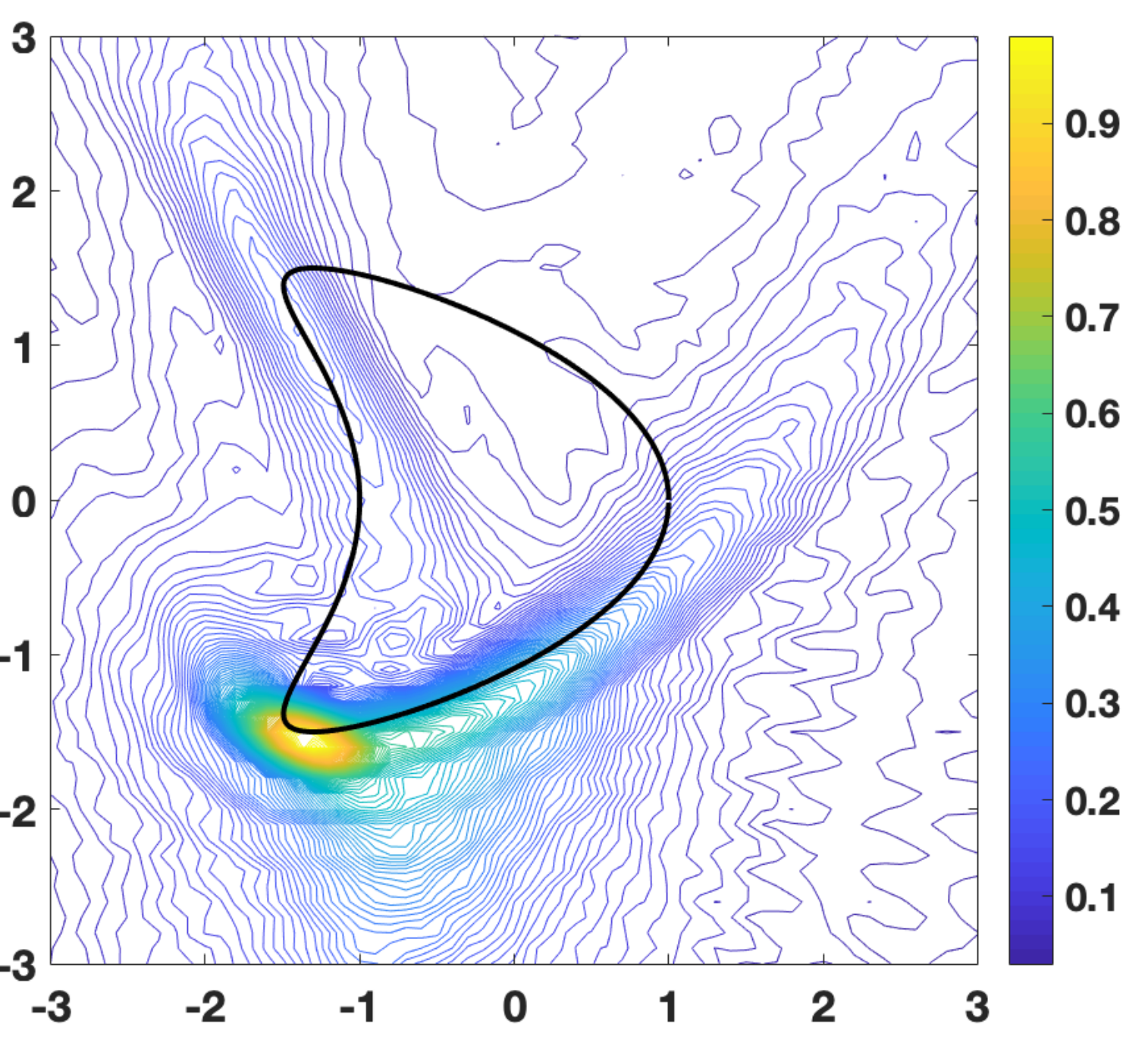}}
\caption{{\bf Example-3.}\, Reconstructions by $I_1^{(3)}$ with fixed incident direction $\theta$ and $32$ observation directions.}\label{Ex32}
\end{figure}

 To improve the result, we use four incident directions, and take the observation directions as follows.
 \be\label{hxhth}
 \hat x\in \left\{
            \begin{array}{ll}
              \Theta_{32} \cap \{(\cos\alpha, \sin\alpha)|\, \alpha\in [3/4\pi, 5/4\pi]\}, &\mbox{if}\, \theta=[1,0],  \\
              \Theta_{32} \cap \{(\cos\alpha, \sin\alpha)|\, \alpha\in [5/4\pi, 7/4\pi]\}, &\mbox{if}\,  \theta=[0,1], \\
              \Theta_{32} \cap \{(\cos\alpha, \sin\alpha)|\, \alpha\in [7/4\pi, 9/4\pi]\}, &\mbox{if}\,  \theta=[-1,0],\\
              \Theta_{32} \cap \{(\cos\alpha, \sin\alpha)|\, \alpha\in [1/4\pi, 3/4\pi]\}, &\mbox{if}\,  \theta=[0,-1].
            \end{array}
          \right.
\en
Fig. \ref{Ex33} gives the result. Clearly, the shape of kite is well reconstructed. The reconstruction is comparable to the results shown in Figure \ref{Ex3} with different data sets.\\

 \begin{figure}[htbp]
   \centering
 \includegraphics[height=2in,width=2in]{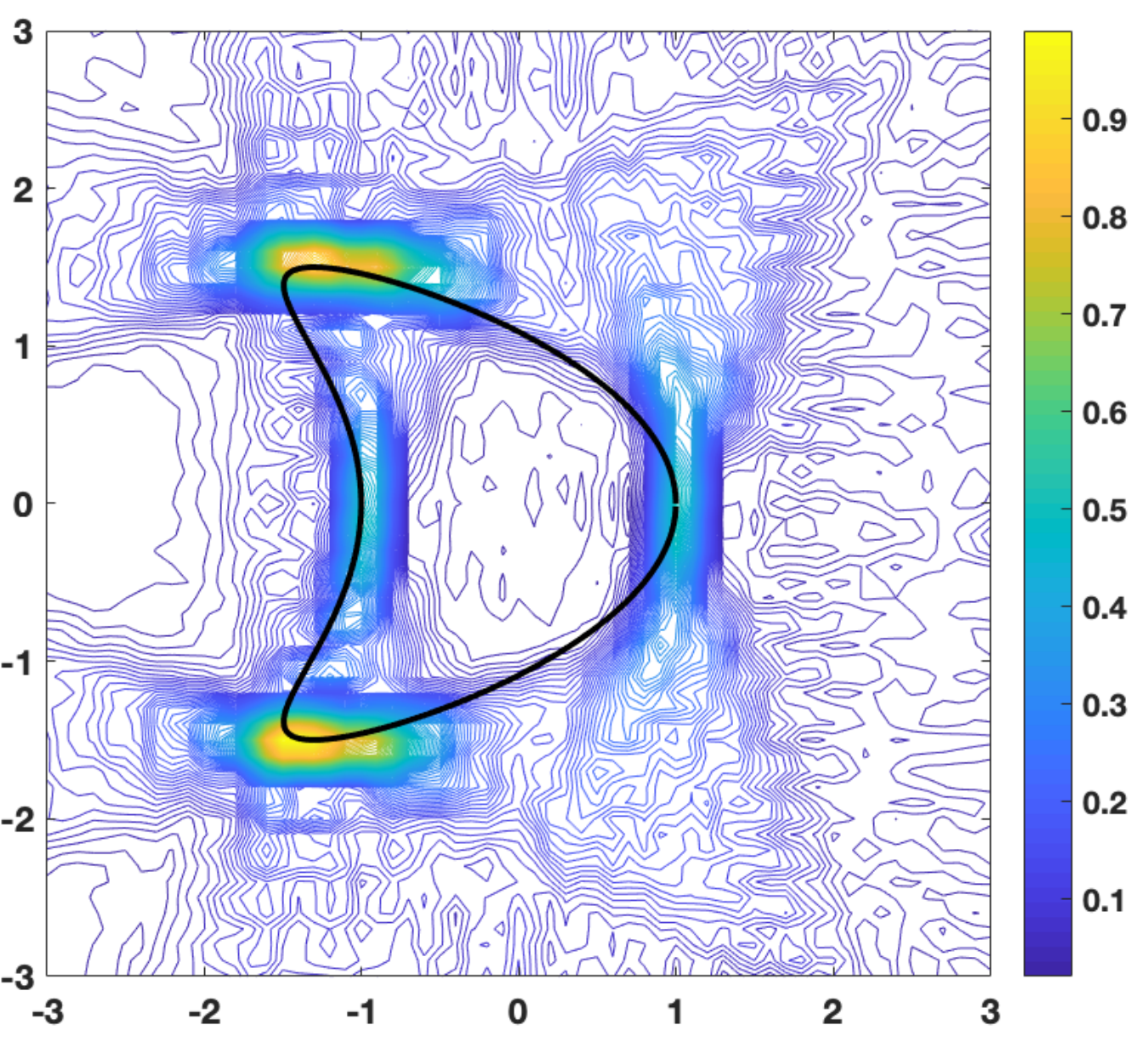}
 \caption{{\bf Example-3.}\, Reconstruction by $I^{(3)}_1$ using $32$ pairs of directions with $(\hx, \hth)$ given in \eqref{hxhth}.}\label{Ex33}
\end{figure}

 \textbf{Example-4:} In this example, the kite with different boundary conditions are considered. We use $I_1^{(1)}$ with 32 incident directions. Fig. \ref{Ex5} gives the results, we choose $\lambda=0.5$ in the impedance boundary and $q=2$ for the penetrable medium.\\

 \begin{figure}[htbp]
  \centering
  \subfigure[\textbf{Sound-hard.}]{
    \includegraphics[height=2in,width=2in]{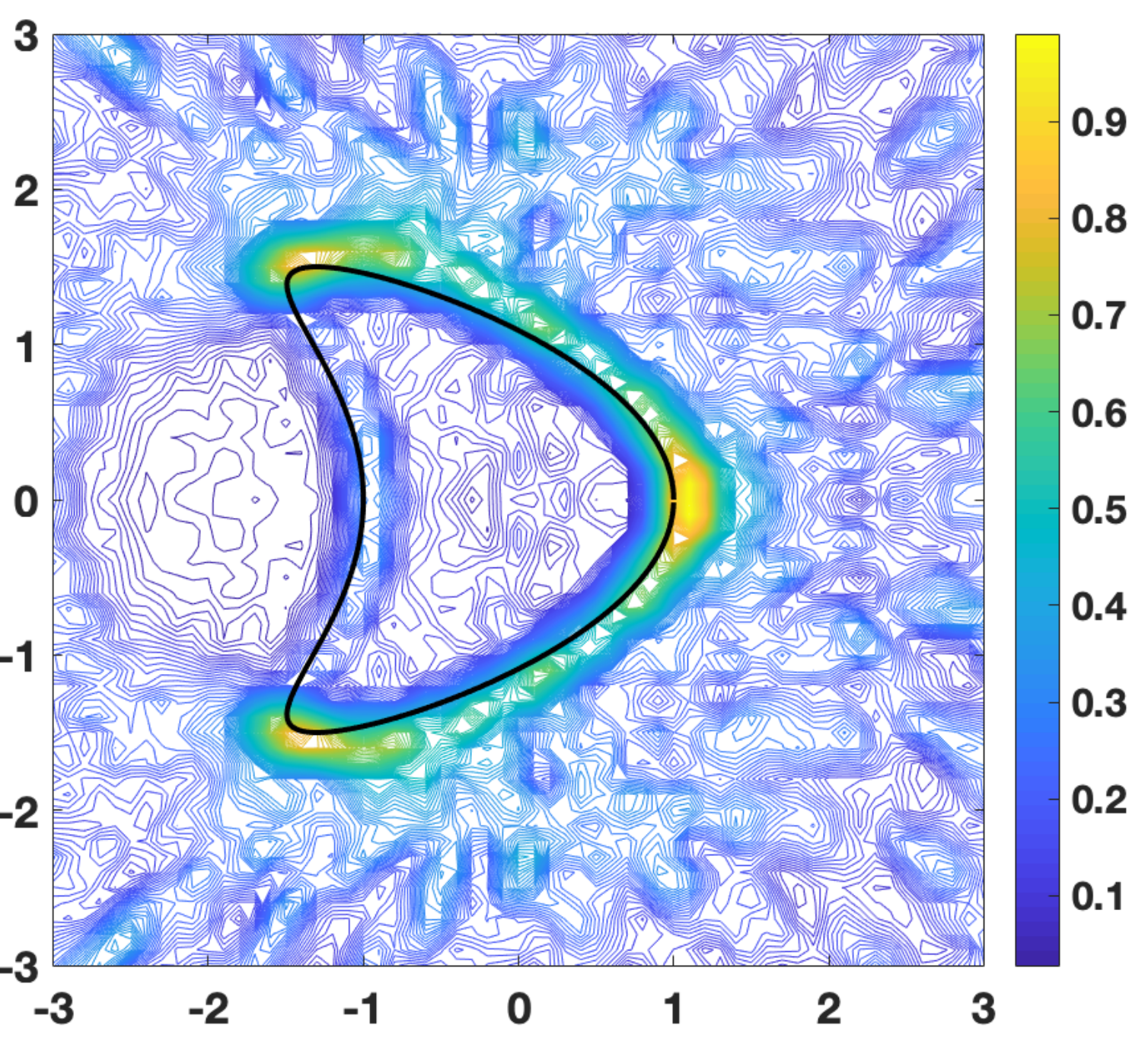}}
  \subfigure[\textbf{Impedance.}]{
    \includegraphics[height=2in,width=2in]{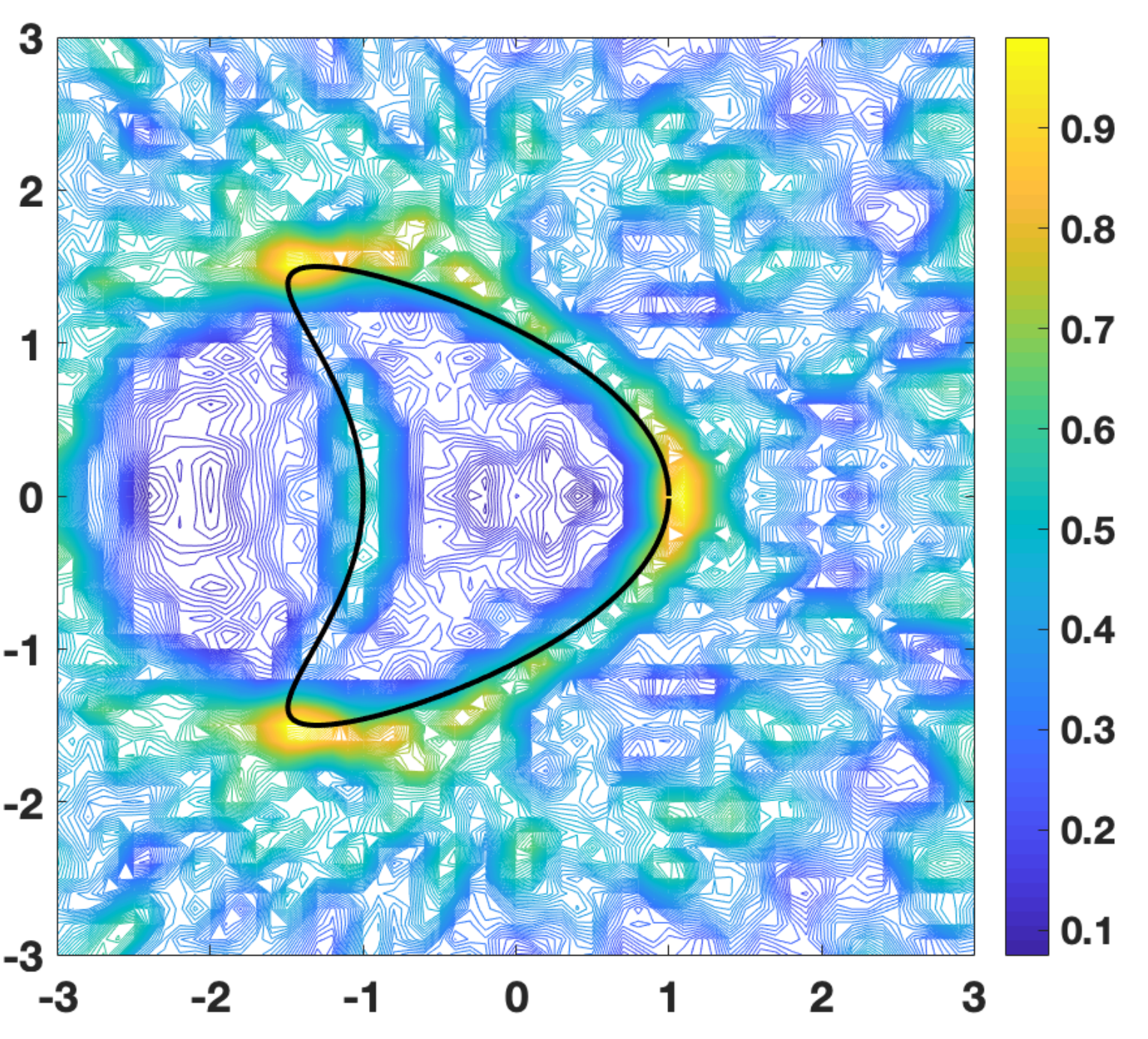}}
     \subfigure[\textbf{Penetrable.}]{
    \includegraphics[height=2in,width=2in]{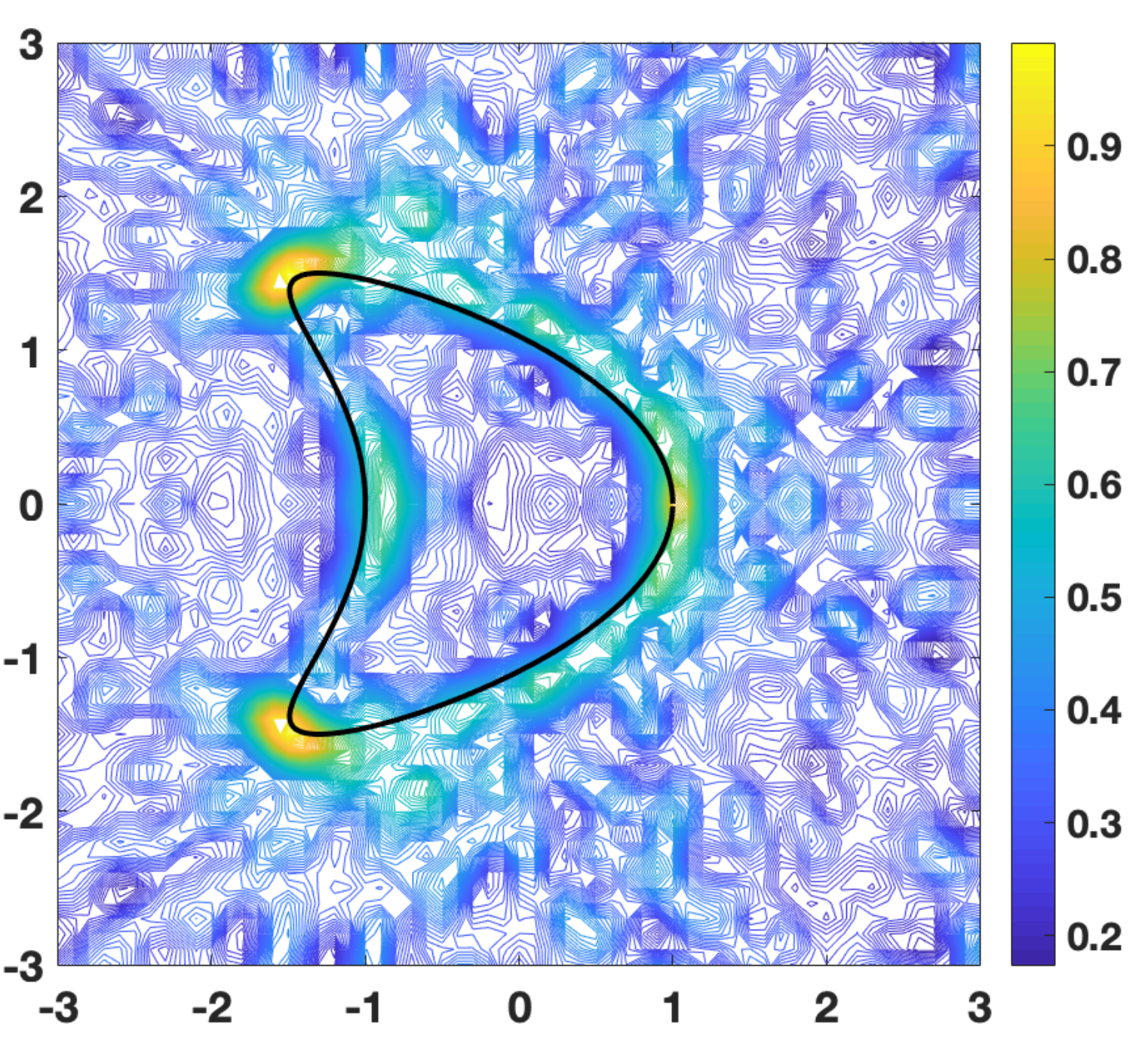}}
\caption{{\bf Example-4.}\, Reconstructions by $I_1^{(1)}$ with $32$ pairs of directions for different boundary conditions.}\label{Ex5}
\end{figure}

\textbf{Example-5:} In this example, we consider the reconstruction of point like scatterers. $160$ equally distributed wave numbers in $[20,100]$ are used.
As suggested in \eqref{I1behaviorpoints}, we take a wider frequency band to ensure that the indicator take the behavior as Dirac delta function.
Fig \ref{Ex6} gives the reconstruction of four points $(1,1),(-1,1),(1,-1),(-1,-1)$ with one or two pairs of directions.
Fig. \ref{Ex7} gives the reconstruction of the word {\bf CAS} with $32$ directions are used. \\

\begin{figure}[htbp]
  \centering
  \subfigure[\textbf{$\hth=[1,0]$.}]{
    \includegraphics[height=2in,width=2in]{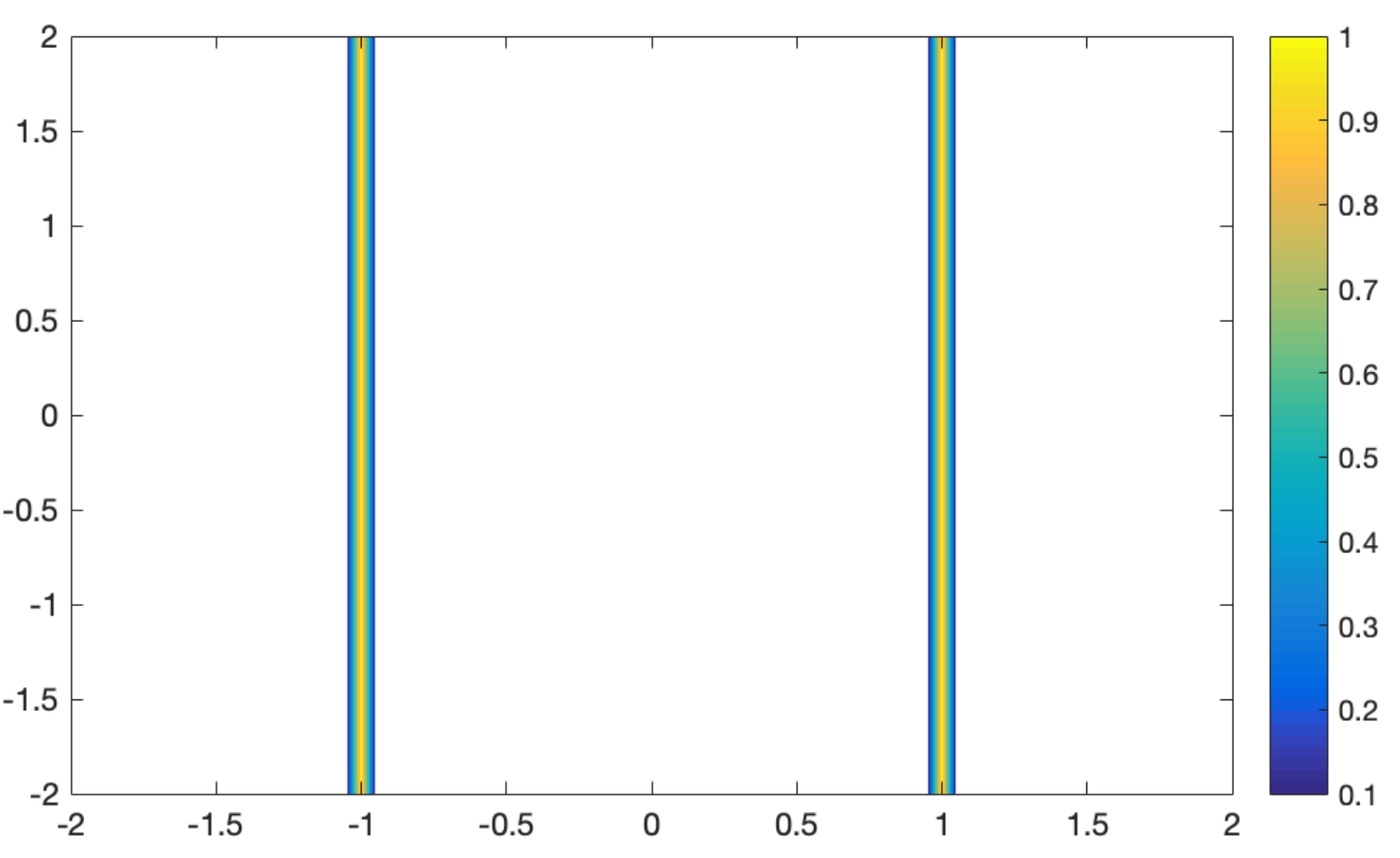}}
  \subfigure[\textbf{$\hth=[0,1]$.}]{
    \includegraphics[height=2in,width=2in]{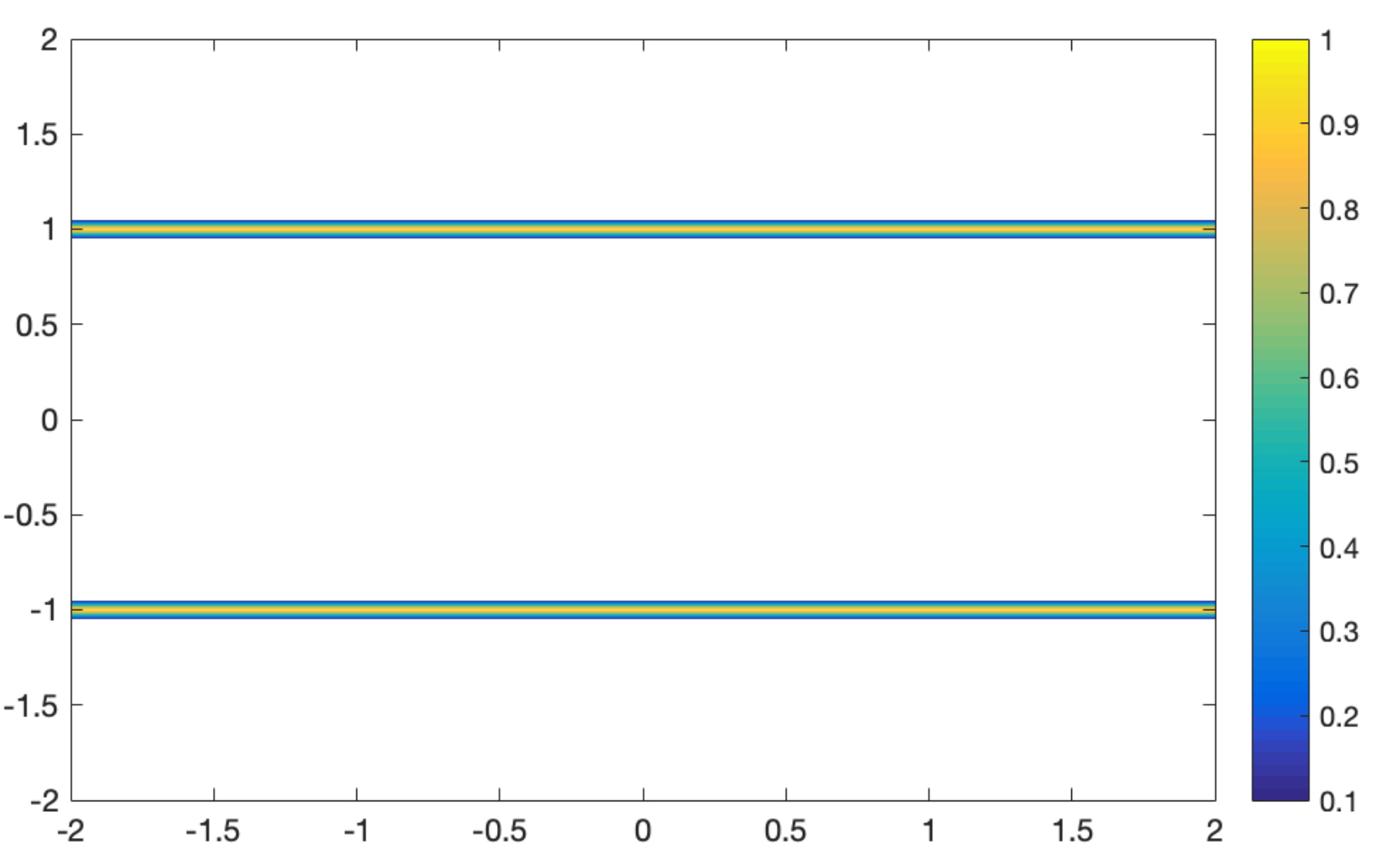}}
  \subfigure[\textbf{$\hth=[1,0],[-1,0]$.}]{
    \includegraphics[height=1.98in,width=2in]{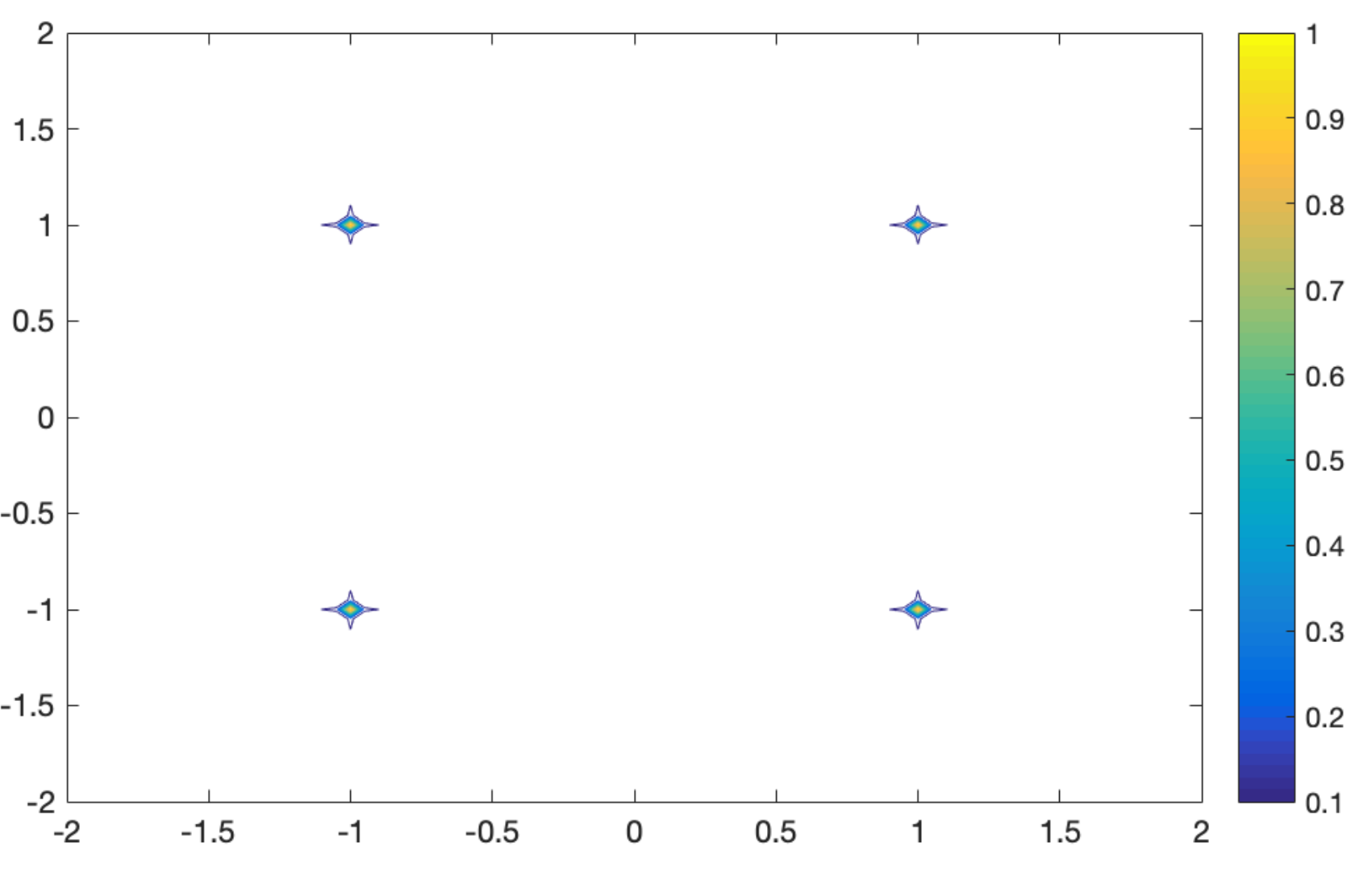}}
  \subfigure[\textbf{$\hth=[1,0]$.}]{
    \includegraphics[height=2in,width=2in]{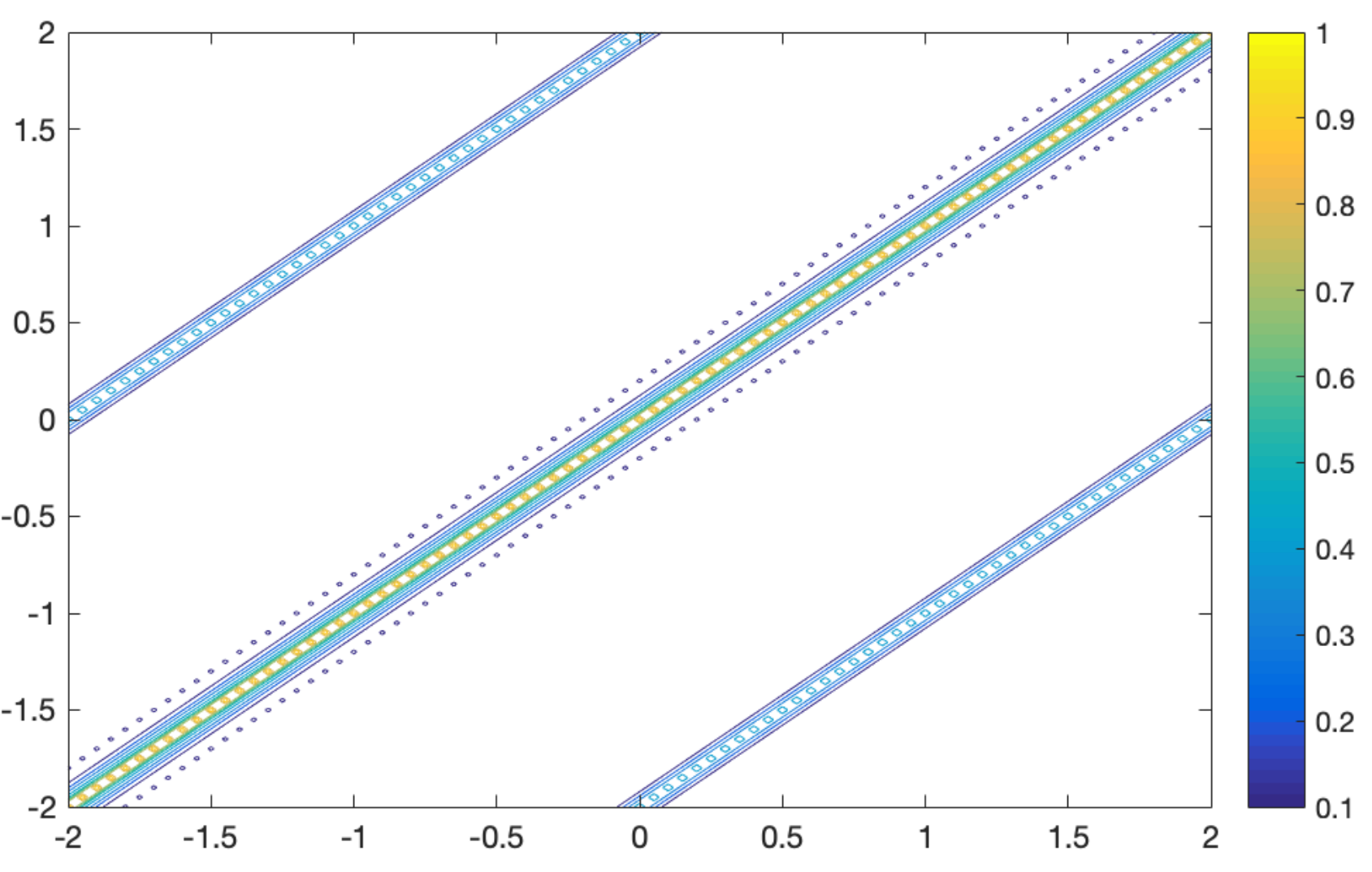}}
  \subfigure[\textbf{$\hth=[\sqrt{2}/2,\sqrt{2}/2]$.}]{
    \includegraphics[height=2in,width=2in]{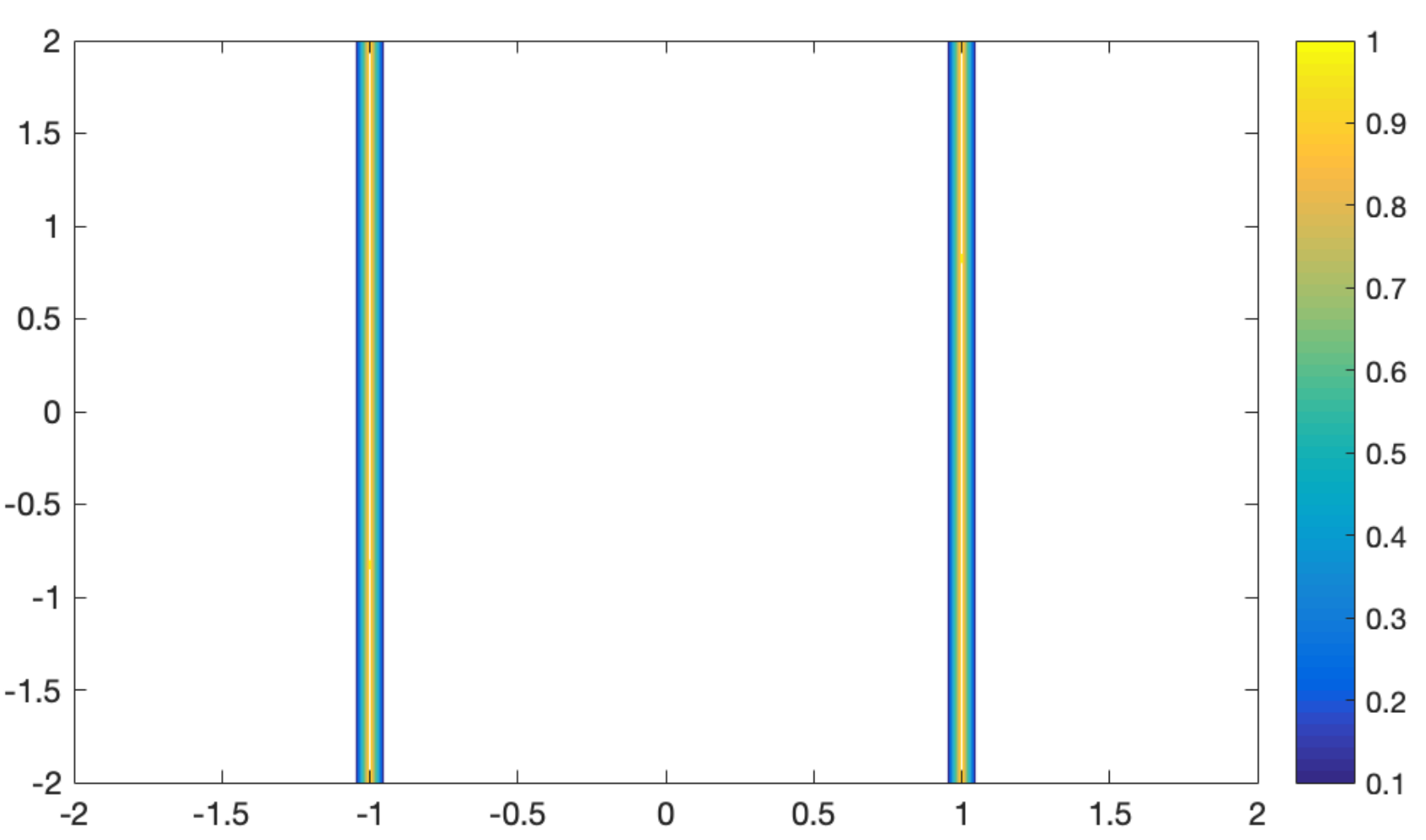}}
  \subfigure[\textbf{$\hth=[1,0],[\sqrt{2}/2,\sqrt{2}/2]$.}]{
    \includegraphics[height=1.98in,width=2in]{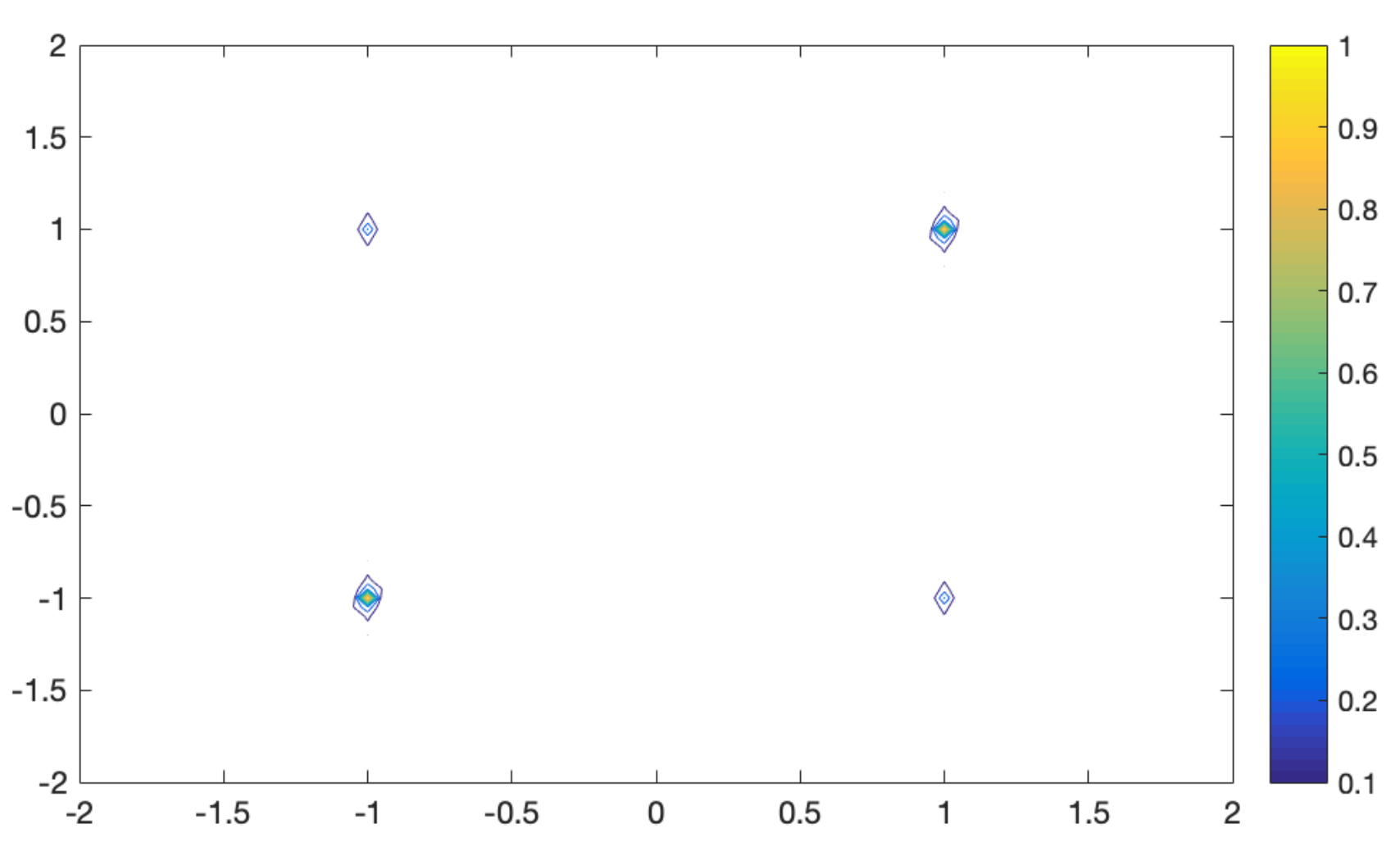}}
\caption{{\bf Example-5.}\, Reconstructions by $I_1^{(1)}$ (top) and $I_1^{(2)}$ (bottom) for four point like scatterers with one or two pairs of directions.}\label{Ex6}
\end{figure}

\begin{figure}[htbp]
  \centering
  \includegraphics[height=2in,width=6in]{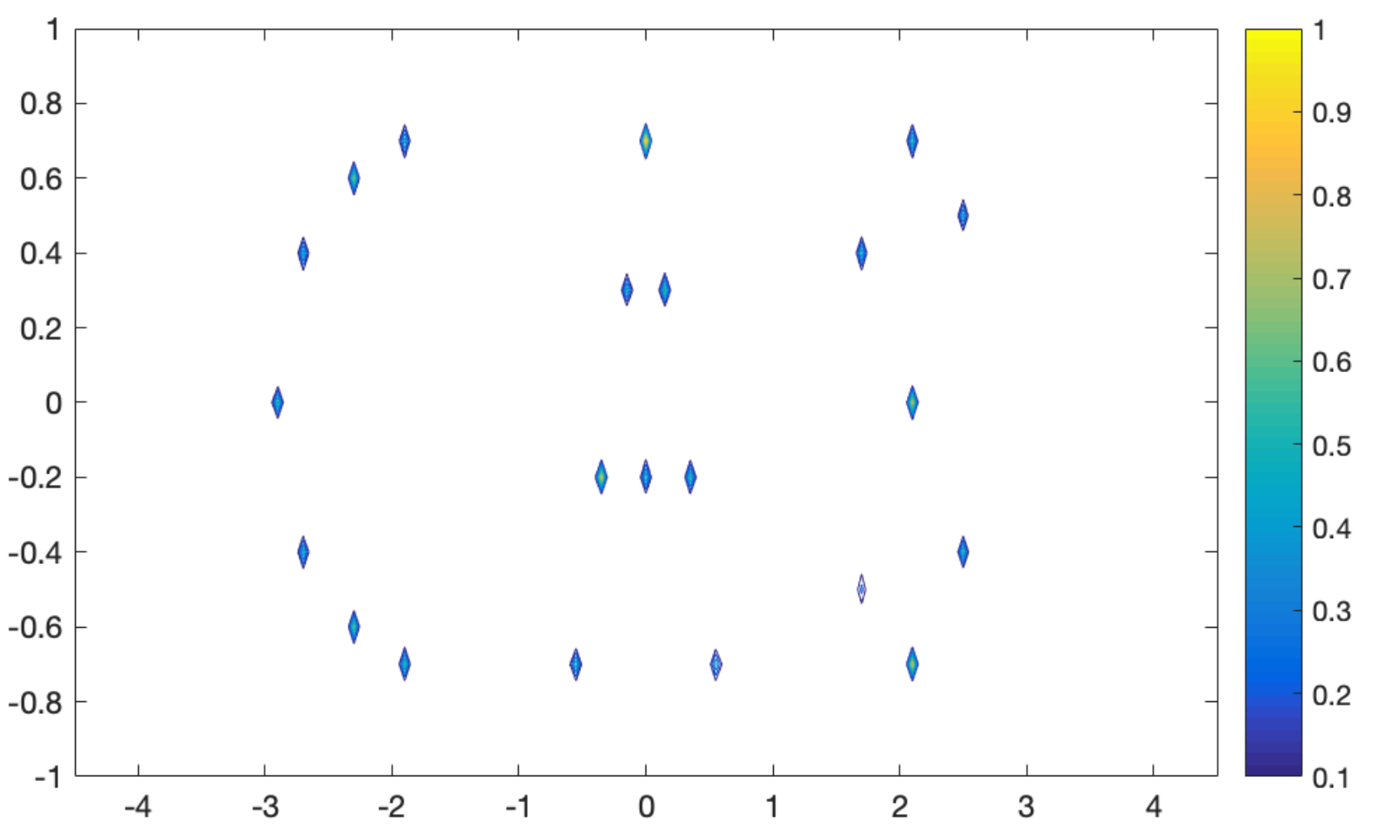}
 \caption{{\bf Example-5.}\, Reconstruction of the word {\bf CAS}, which is the abbreviation for Chinese Academy of Sciences.}\label{Ex7}
\end{figure}

\textbf{Example-6:}
In this example, we consider the scatterers with multiple multiscalar components. The first underlying scatterer is a sound soft kite with $(a,b)=(0,0)$ and a small sound soft circle with $(a,b)=(2.5,2.5)$ and $r=0.1$. To enhance the resolution, $40$ wave numbers are used.
All the numerical results shown in Figures \ref{Ex8} are satisfactory.

\begin{figure}[htbp]
  \centering
  \subfigure[\textbf{$\hth=[1,0]$.}]{
    \includegraphics[height=2in,width=2in]{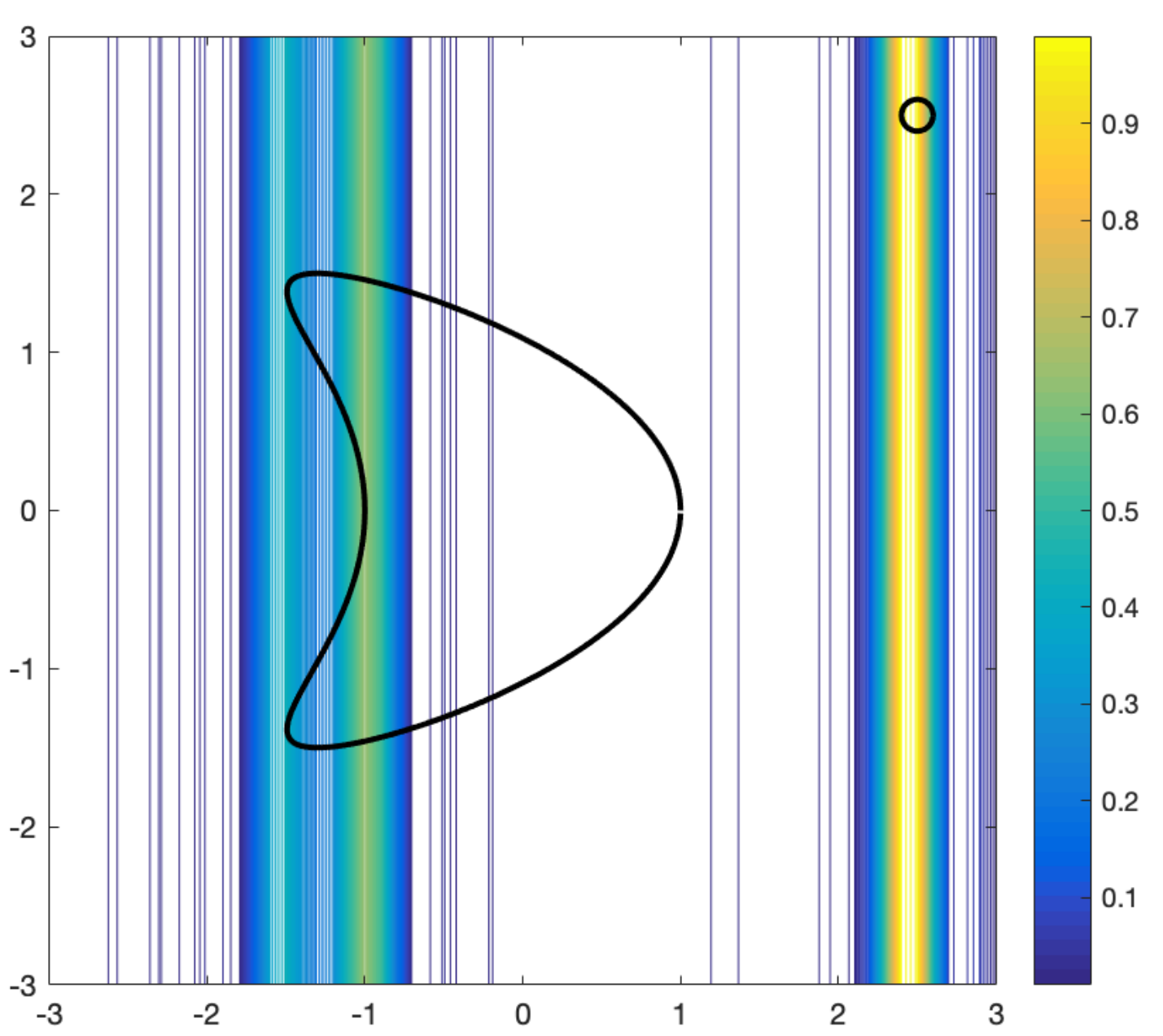}},
  \subfigure[\textbf{4 directions..}]{
    \includegraphics[height=2in,width=2in]{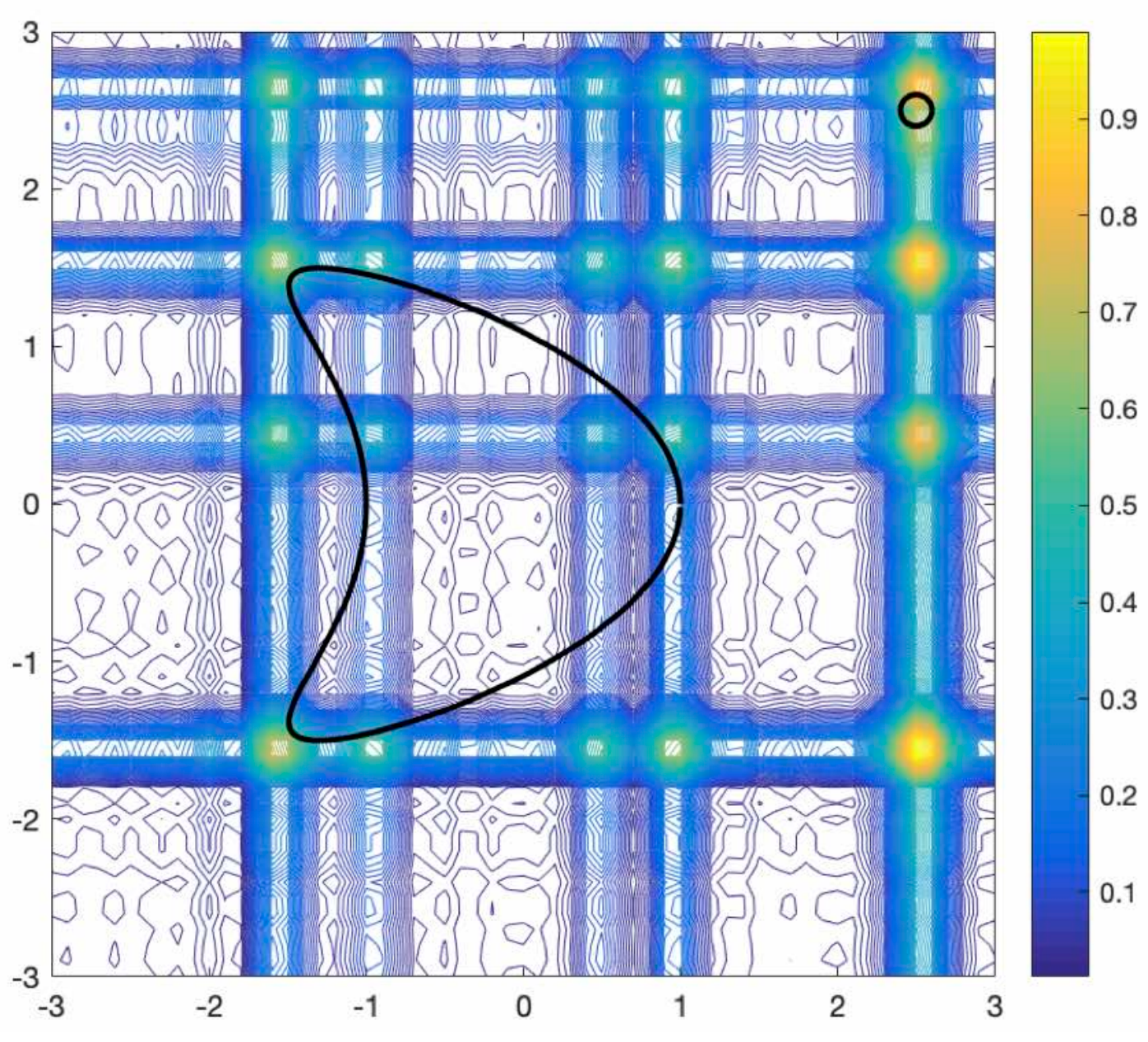}}
  \subfigure[\textbf{32 directions.}]{
    \includegraphics[height=1.98in,width=2in]{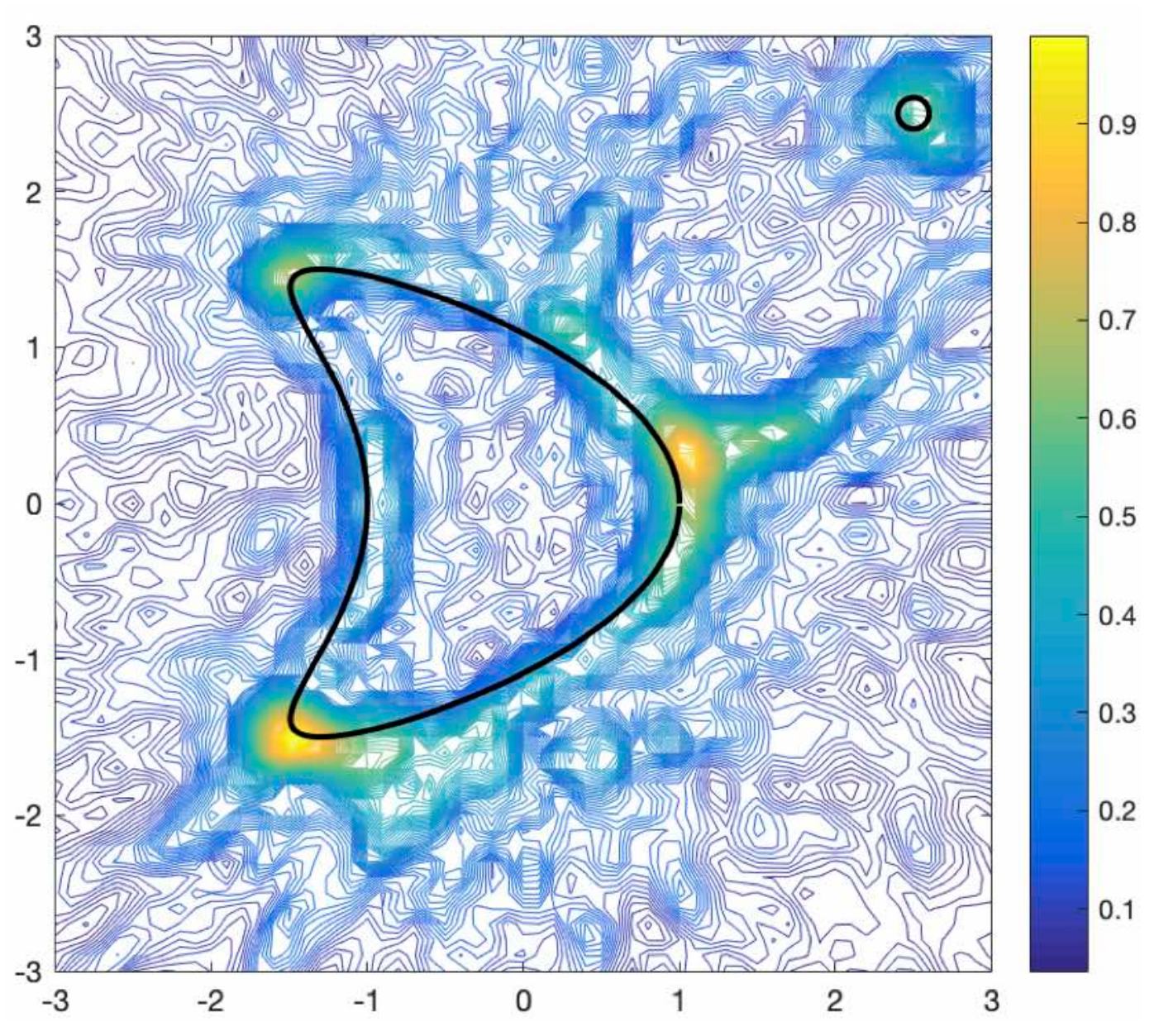}}
\caption{{\bf Example-6.}\, Reconstructions by $I_1^{(1)}$ for multiscalar scatterers.}\label{Ex8}
\end{figure}

We also consider an example with both extended component and point like components. The underlying scatterer is a sound soft kite with $(a,b)=(0,0)$ and three point like scatterers located at $(2.5,2),(2.5,0)$ and $(2.5,-2)$, $40$ wave numbers are used in this example.  Fig. \ref{Ex9} gives the reconstruction.

\begin{figure}[htbp]
  \centering
  \subfigure[\textbf{$\hth=[1,0]$.}]{
    \includegraphics[height=2in,width=2in]{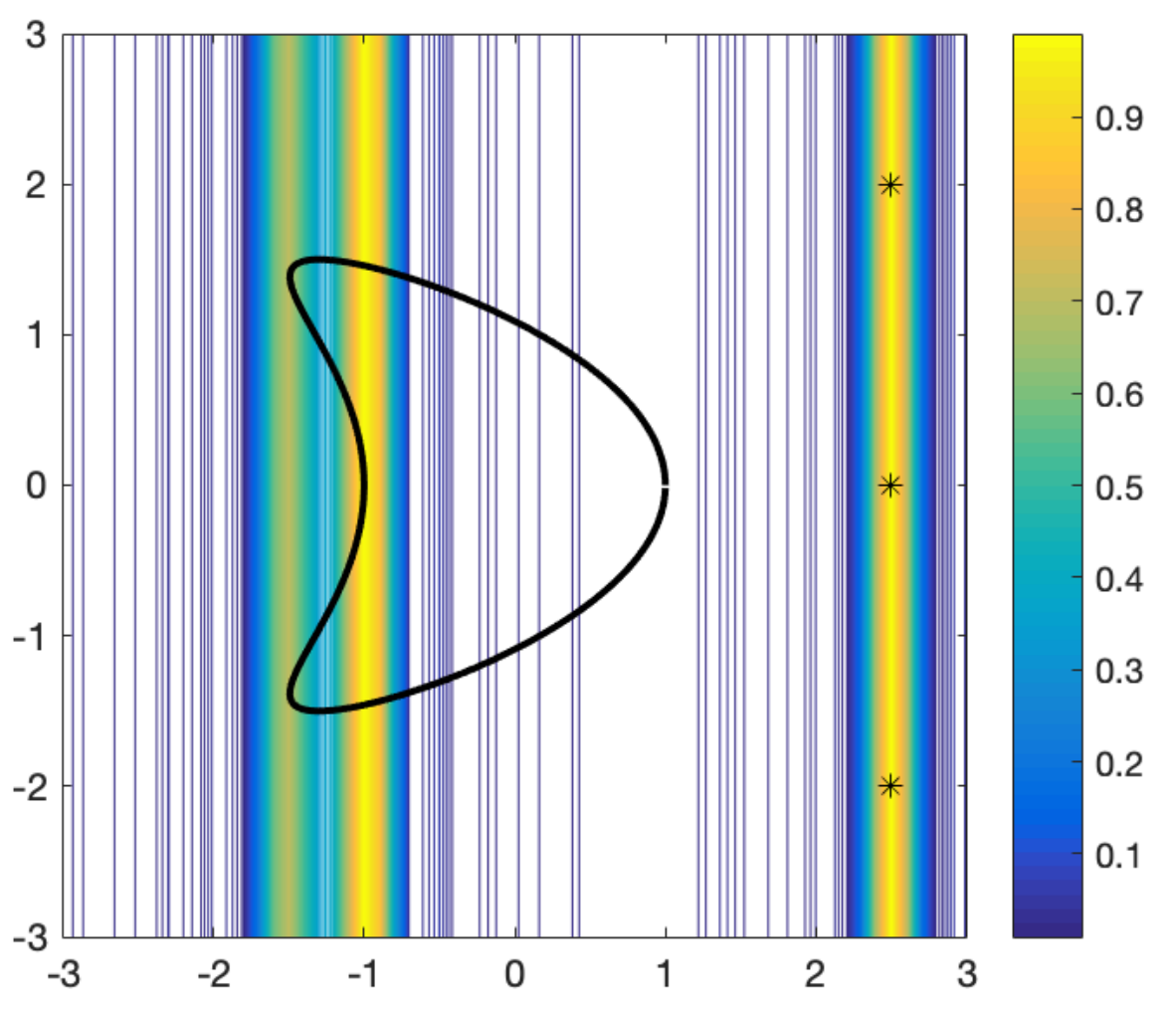}},
  \subfigure[\textbf{4 directions..}]{
    \includegraphics[height=2in,width=2in]{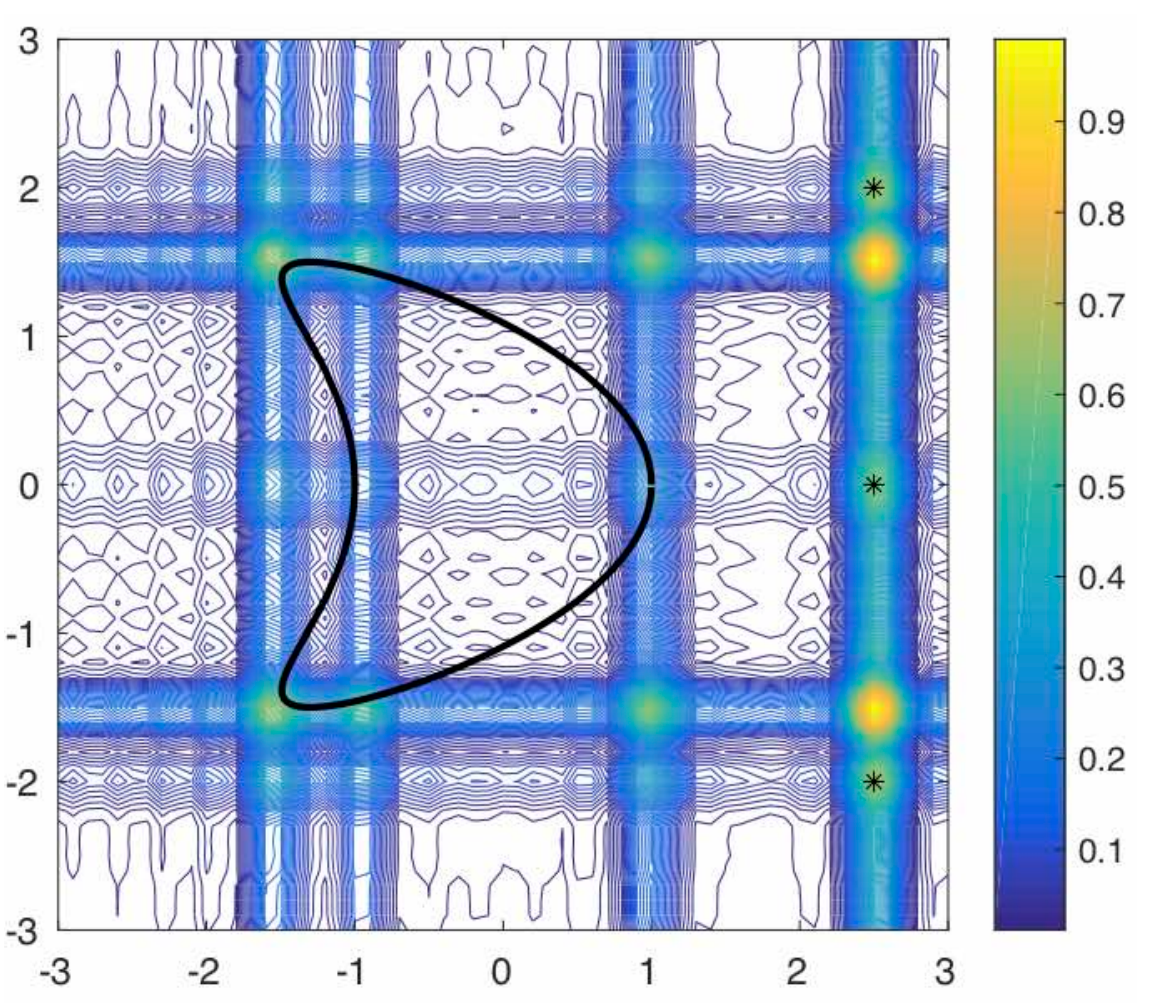}}
  \subfigure[\textbf{32 directions.}]{
    \includegraphics[height=1.98in,width=2in]{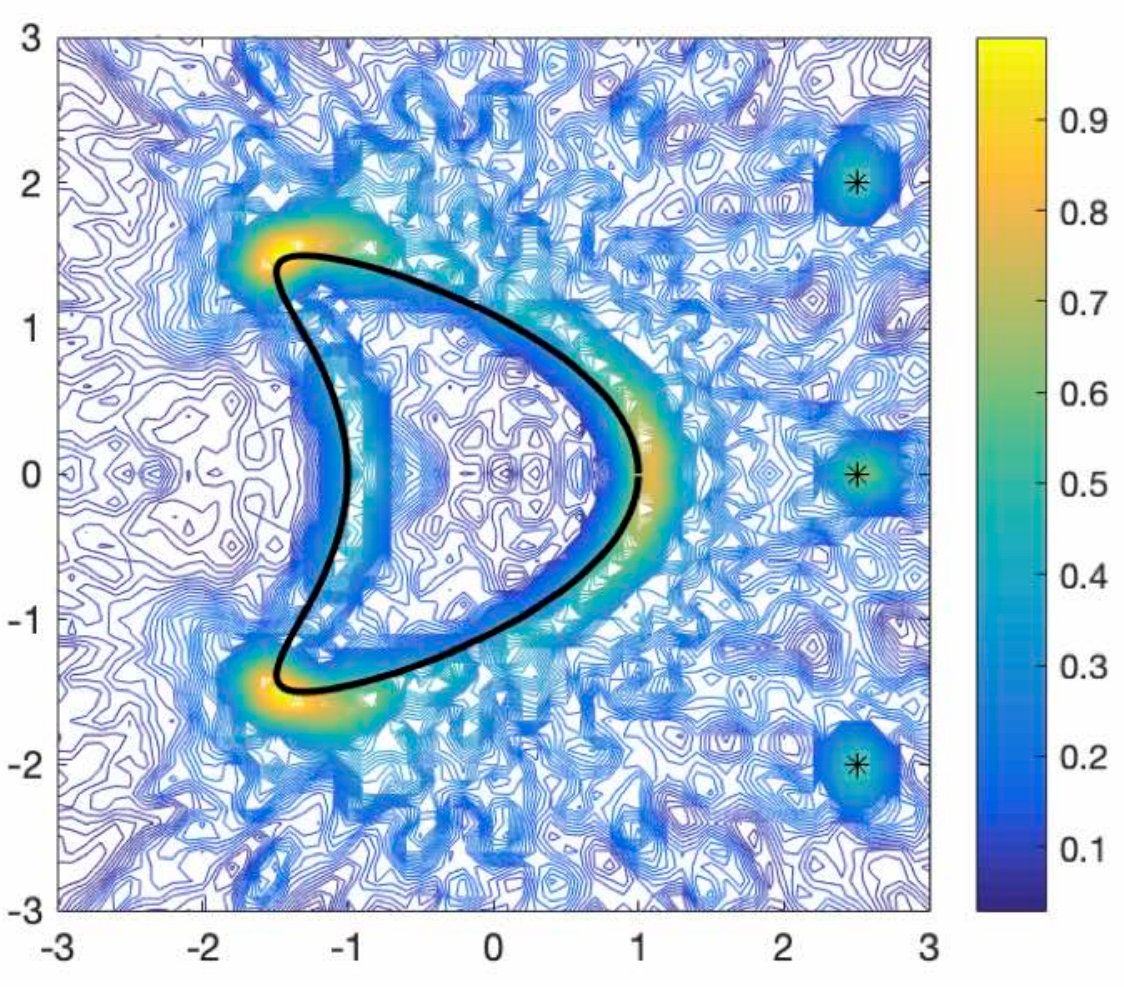}}
\caption{{\bf Example-6.}\, Reconstructions by $I_1^{(1)}$ with different number of incident directions.}\label{Ex9}
\end{figure}

\section{Concluding remarks}

In this paper we propose a sampling method for shape identification in inverse acoustic scattering problem with multi-frequency backscattering far field patterns at sparse observation directions. To our best knowledge, this is the first numerical method for inverse obstacle/medium scattering problems with sparse backscattering data.
The theory foundation has been established based on the weak scattering approximation and the Kirchhoff approximation. In particular, we find that at most $2n$ observation directions are enough to approximately reconstruct a hull of the underlying scatterers. The numerical results show that both the location and shape of the scatterer can be well captured with the increase of the number of the observation directions, even the underlying scatterer has concave part or multiple multiscalar components.

Similar techniques can also be applied to inverse scattering of elastic waves or electromagnetic waves, which shall be addressed in a forthcoming work.

\section*{Acknowledgement}
The research of X. Ji is supported by the NNSF of China under grants 91630313 and 11971468,
and National Centre for Mathematics and Interdisciplinary Sciences, CAS.
The research of X. Liu is supported by the NNSF of China under grant 11971701, and the Youth Innovation Promotion Association, CAS.

\bibliographystyle{SIAM}

\end{document}